\documentclass[10pt]{article}

\usepackage{amssymb}   
\usepackage{amsthm}    
\usepackage{amsmath}   
\usepackage{stmaryrd}  
\usepackage{titletoc}  
\usepackage{mathrsfs}  
\usepackage{graphicx}
\usepackage{xcolor}
\usepackage[toc,page,title,titletoc,header]{appendix}
\usepackage{ulem}

\newlength{\defbaselineskip}
\newcommand{\setlinespacing}[1]%
           {\setlength{\baselineskip}{#1 \defbaselineskip}}

\makeatletter\@addtoreset{equation}{section} \makeatother
 \allowdisplaybreaks

\newtheorem{theorem}{Theorem}[section]
\newtheorem{lemma}[theorem]{Lemma}

\theoremstyle{definition}
\newtheorem{definition}[theorem]{Definition}

\newtheorem{proposition}[theorem]{Proposition}
\newtheorem{corollary}[theorem]{Corollary}

\theoremstyle{remark}
\newtheorem{remark}[theorem]{Remark}

\numberwithin{equation}{section}

\begin{document}

\title{Mean Field Games with Singular Controls\thanks{Financial support by the Berlin Mathematical School (BMS) and the SFB 649 {\sl Economic Risk} is gratefully acknowledged. We thank an anonymous referee for many comments and suggestions that greatly helped to improve the presentation of the results. }}

\author{Guanxing Fu\footnote{Department of Mathematics, and Berlin Mathematical School, Humboldt-Universit\"at zu Berlin
         Unter den Linden 6, 10099 Berlin, Germany; email: fuguanxing725@gmail.com}~ and Ulrich Horst\footnote{Department of Mathematics, and School of Business and Economics, Humboldt-Universit\"at zu Berlin
         Unter den Linden 6, 10099 Berlin, Germany; email: horst@math.hu-berlin.de}}
%





\maketitle

\begin{abstract}
This paper establishes the existence of relaxed solutions to mean field games (MFGs for short) with singular controls. 
 We also prove approximations of solutions results for a particular class of MFGs with singular controls by solutions, respectively control rules, for MFGs with purely regular controls. Our existence and approximation results strongly hinge on the use of the Skorokhod $M_1$ topology on the space of c\`adl\`ag functions.
\end{abstract}



\section{Introduction and overview}

Starting with the seminal papers \cite{huang-malhame-caines-2006,lasry-lions-2007}, the analysis of \textit{mean field games} (MFGs) has received considerable attention in the stochastic control and financial mathematics literature. In a standard MFG, each player $i \in \{1, ..., N\}$ chooses an action from a given set of admissible controls that minimizes a cost functional of the form
    \begin{equation}\label{cost-fun}
        J^i(u)=E\left[\int_0^Tf(t,X^i_t,\bar{\mu}^N_t,u^i_t)dt+g(X^i_T,\bar{\mu}^N_T)\right]
    \end{equation}
subject to the state dynamics
   \begin{equation}\label{state-without-singular}
    \left\{\begin{array}{ll}
        dX^i_t=b(t,X^i_t,\bar{\mu}^N_t,u^i_t)\,dt+\sigma(t,X^i_t,\bar{\mu}^N_t,u^i_t)\,dW^i_t,\\
        X^i_0=x_0
        \end{array} .
        \right.
    \end{equation}
Here $W^1, ..., W^N$ are independent Brownian motions defined on some underlying filtered probability space, $u=(u^1,\cdots,u^N)$, $u^i=(u^i_t)_{t \in [0,T]}$ is an adapted stochastic process, the {\it action} of player $i$, and $\bar{\mu}^N_t:=\frac{1}{N}\sum_{j=1}^N\delta_{X^j_t}$ denotes the empirical distribution of the individual players' states at time $t \in [0,T]$. In particular, all players are identical ex ante and each player interacts with the other players only through the empirical distribution of the state processes.

The existence of {\it approximate Nash equilibria} in the above game for large populations has been established in \cite{carmona-delarue-2013, huang-malhame-caines-2006} using a representative agent approach. In view of the independence of the Brownian motions the idea is to first approximate the dynamics of the empirical distribution by a deterministic measure-valued process, and to consider instead the optimization problem of a representative player that takes the distribution of the states as given, and then to solve the fixed-point problem of finding a measure-valued process such that the distribution of the representative player's state process $X$ under her optimal strategy coincides with that process.\footnote{The idea of decoupling local from global dynamic in large population has been applied to equilibrium models of social interaction in e.g. \cite{HS1, HS2}.}

Following the representative agent approach, a MFG can then be formally described by a coupled optimization and fixed point problem of the form:
\begin{equation}\label{model-MFG}
        \left\{
        \begin{array}{ll}
        1.&\textrm{fix a deterministic function }t\in[0,T]\mapsto\mu_t\in\mathcal{P}(\mathbb{R}^d);\\
        2.&\textrm{solve the corresponding stochastic control problem}:\\
         &\inf_{u} E\left[\int_0^Tf(t,X_t,\mu_t,u_t)\,dt+g(X_T,\mu_T)\right],\\
        &\textrm{subject to}\\
        &dX_t =b(t,X_t,{\mu}_t,u_t)\,dt+\sigma(t,X_t,{\mu}_t,u_t)\,dW_t\\
        &X_{0} =x_0,\\
        3.&\textrm{solve the fixed point problem:} ~Law(X) =\mu,
        \end{array}
        \right.
    \end{equation}
where $\mathcal{P}(\mathbb{R}^d)$ is the space of probability measures on $\mathbb{R}^d$ and $Law(X)$ denotes the law of the process $X$.

There are essentially three approaches to solve mean field games. In their original paper \cite{lasry-lions-2007}, Lasry and Lions followed an analytic approach. They analyzed a coupled forward-backward PDE system, where the backward component is the Hamiltion-Jacobi-Bellman equation arising from the representative agent's optimization problem, and the forward component is a Kolmogorov-Fokker-Planck equation that characterizes the dynamics of the state process.

A second, more probabilistic, approach was introduced by Carmona and Delarue in \cite{carmona-delarue-2013}. Using a maximum principle of Pontryagin type, they showed that the fixed point problem reduces to solving a McKean-Vlasov forward-backward SDEs (FBSDEs for short). Other results based on probabilistic approaches include \cite{ahuja-2016,bensoussan-sung-yam-yung-2016,carmona-delarue-lachapelle}. Among them,  \cite{bensoussan-sung-yam-yung-2016,carmona-delarue-lachapelle} consider linear-quadratic MFGs, while \cite{ahuja-2016,carmona-zhu-2016} consider MFGs with common noise and with major and minor players, respectively. A class of MFGs in which the interaction takes place both through the state dynamics and the controls has recently been introduced in \cite{carmona-lacker-2015}. In that paper the weak formulation, or martingale optimality principle, is used to prove the existence of a solution.

A \textit{relaxed solution} concept to MFGs was introduced by Lacker in \cite{Lacker-2015}. Considering MFGs from a more game-theoretic perspective, the idea is to search for equilibria in relaxed controls (``mixed strategies'') by first establishing the upper hemi-continuity of the representative agent's best response correspondence to a given $\mu$ using Berge's maximum theorem, and then to apply the Kakutani-Fan-Glicksberg fixed point theorem in order to establish the existence of some measure-valued process $\mu^*$ such that the law of the agent's state process under a best response to $\mu^*$ coincides with that process.
Relaxed controls date back to Young \cite{young-1937}. They were later applied to stochastic control in, e.g. \cite{haussmann-lepeltier-1990,Haussmann-Suo-1995,Karoui-Nguyen-Picque-1990}, to MFGs in \cite{Lacker-2015}, and to MFGs with common noise in \cite{carmona-delarue-lacker-2016}. 

Applications of MFGs range from models of optimal  exploitation of exhaustible resources \cite{chan-sircar-2015,gueant-lasry-linos-2010} to systemic risk \cite{carmona-fouque-sun}, and from principal-agent problems \cite{elie-mastrolia-possamai-2016} to problems of optimal trading under market impact \cite{carmona-lacker-2015,jaimungal-nourian-2015}. Motivated by possible applications to optimal portfolio liquidation under strategic interaction that allow for both block trades and absolutely continuous trades as in \cite{HN}, this paper provides a probabilistic framework for analyzing MFGs with singular controls. {Extending \cite{Lacker-2015}, we consider MFGs with singular controls of the form}
\begin{equation}\label{model-MFGs-singular-control}
        \left\{
        \begin{array}{ll}
        1.&\textrm{fix a deterministic function }t\in[0,T]\mapsto\mu_t\in\mathcal{P}(\mathbb{R}^d);\\
        2.&\textrm{solve the corresponding stochastic singular control problem}:\\
         &\inf_{u,Z} E\left[\int_0^Tf(t,X_t,\mu_t,u_t)\,dt+g(X_T,\mu_T)+\int_0^Th(t)\,dZ_t\right],\\
        &\textrm{subject to}\\
        &dX_t =b(t,X_t,{\mu}_t,u_t)\,dt+\sigma(t,X_t,{\mu}_t,u_t)\,dW_t+c(t)\,dZ_t,\\
       3.&\textrm{solve the fixed point problem:~}
        Law(X) =\mu,
        \end{array}
        \right.
    \end{equation}
where $u=(u_t)_{t \in [0,T]}$ is the {\it regular control}, and $Z=(Z_t)_{t \in [0,t]}$ is the {\it singular control}.
When singular controls are admissible, the state process no longer takes values in the space of continuous functions, but rather in the Skorokhod space $\mathcal{D}(0,T)$ of all c\`adl\`ag functions. The key is then to identify a suitable topology on the Skorokhod space with respect to which the compactness and continuity assumptions of the maximum and the fixed-point theorems are satisfied.

There are essentially three possible topologies on the space of c\`adl\`ag functions: the (standard) Skorokhod $J_1$ topology ($J_1$ topology for short), the Meyer-Zheng topology (or pseudo-path topology), and the Skorokhod $M_1$ topology ($M_1$ topology for short). The $M_1$ topology seems to be the most appropriate one for our purposes. First, the set of bounded singular controls is compact in the $M_1$ topology but not in the $J_1$ topology. Second, there is no explicit expression for the metric corresponding to Meyer-Zheng topology. In particular, one cannot bound the value of a function at given points in time by the Meyer-Zheng topology. Third, the $M_1$ topology has better continuity properties than the $J_1$ topology. For instance, it allows for an approximation of discontinuous functions by continuous ones. This enables us to approximate solutions to certain classes of MFGs with singular controls by solutions to MFGs with only regular controls. Appendix B summarizes useful properties of the $M_1$ topology; for more details, we refer to the textbook of Whitt \cite{Whitt-2002}.

To the best of our knowledge, ours is the first paper to establish the existence of solutions results to MFGs with singular controls\footnote{The recent paper \cite{GL-2017} only considers absolutely continuous singular controls. Our notion of singular controls is more general.}. As a byproduct, we obtain a new proof for the existence of optimal (relaxed) controls for the corresponding class of stochastic singular control problems. A similar control problem, albeit with a trivial terminal cost function has been analyzed in \cite{Haussmann-Suo-1995}. While the methods and techniques applied therein can be extended to non-trivial terminal cost functions after a modification of the control problem, they cannot be used to prove existence of equilibria in MFGs. In fact, in \cite{Haussmann-Suo-1995}, it is assumed that the state space $\mathcal{D}(0,T)$ is endowed with Meyer-Zheng topology, and that the spaces of admissible singular and regular controls are endowed with the topology of weak convergence and the stable topology, respectively. With this choice of topologies the continuity of cost functional and the upper-hemicontinuity of distribution of the representative agent's state process under the optimal control w.r.t. to a given process $\mu$ cannot be established. As a second byproduct we obtain a novel existence of solutions result for a class of McKean-Vlasov stochastic singular control problems. 

Our second main contributions are two approximation results that allow us to approximate solutions to a certain class of MFGs with singular controls by the solutions to MFGs with only regular controls. The approximation result, too, strongly hinges on the choice of the $M_1$ topology.

The rest of this paper is organized as follows: in Section 2, we recall the notion of relaxed controls for singular stochastic control problems, introduce MFGs with singular controls and state our main existence of solutions result. The proof is given in Section 3. In Section 4, we state and prove two approximation results for MFGs with singular controls by MFGs with regular controls. Appendix A recalls known results and definitions that are used throughout this paper. Append B reviews key properties of the $M_1$ topology.


\section{Assumptions and the main results}\label{definition-assumption}
In this section we introduce MFGs with singular controls and state our main existence of solutions result. For a metric space $(E,\varrho)$ we denote by $\mathcal{P}_p(E)$ the class of all probability measures on $E$ with finite moment of $p$-th order. For $p=0$ we write $\mathcal{P}(E)$ instead of $\mathcal{P}_0(E)$. The set $\mathcal{P}_p(E)$ is endowed with the Wasserstein distance $\mathcal{W}_{p,(E,\varrho)}$; see Definition \ref{Wasserstein-space-metric}. For a given interval $\mathbb{I}$ we denote by  $\mathcal{D}(\mathbb{I})$ the Skorokhod space of all $\mathbb{R}^d$-valued c\`adl\`ag functions on $\mathbb{I}$, by $\mathcal{A}(\mathbb{I}) \subset \mathcal{D}(\mathbb{I})$ the subset of nondecreasing functions, by $\mathcal{C}(\mathbb I)\subset \mathcal{D}(\mathbb{I})$ the subset of continuous functions, and by $\mathcal{U}(\mathbb I)$ the set of all measures on $\mathbb{I} \times U$ for some metric space $U$, whose first marginal is the Lebesgue measure on $\mathbb I$, and whose second marginal belongs to $\mathcal{P}(U)$. For reasons that will become clear later we identify processes on $[0,T]$ with processes on the whole real line. For instance, we identify the space $\mathcal{D}(0,T)$ with the space
  \[
        \widetilde{\mathcal{D}}_{0,T}(\mathbb{R})=\{x\in\mathcal{D}(\mathbb{R}):x_t=0\textrm{ if }t<0\textrm{ and }x_t=x_T \textrm{ if }t>T\}.
    \]
    Likewise, we denote by $\widetilde{\mathcal{A}}_{0,T}(\mathbb{R})$ and $\widetilde{\mathcal{C}}_{0,T}(\mathbb{R})$ the subspace of $\widetilde{\mathcal{D}}_{0,T}(\mathbb{R})$ with nondecreasing and continuous paths, respectively. Moreover, we denote by $\widetilde{\mathcal U}_{0,T}(\mathbb R)$ all measures $q(dt,du)$ on $\mathbb R\times U$ whose restriction to $[0,T]$ belongs to $\mathcal U(0,T)$, and whose restrictions to $(-\infty,0)$ and $(T,\infty)$ are of the form $q(dt,du)=\delta_{\widetilde{u_0}}(du)dt$ and $q(dt,du)=\delta_{\widetilde{u_T}}(du)dt$ for some $\widetilde{u_0}\in U$ and $\widetilde{u_T}\in U$, respectively. We occasionally drop the subscripts $0$ and $T$ if there is no risk of confusion.
    Throughout this paper, $C > 0$ denotes a generic constant that may vary from line to line.

\subsection{Singular stochastic control problems}
Before introducing MFGs with singular controls, we informally review stochastic singular control problems of the form:
\begin{equation}\label{model-stochastic-singular-control}
        \left\{
        \begin{array}{ll}
        \inf_{u,Z} E\left[\int_0^Tf(t,X_t,u_t)\,dt+g(X_T)+\int_0^Th(t)\,dZ_t\right],\\
        \textrm{subject to}\\
        dX_t=b(t,X_t,u_t)\,dt+\sigma(t,X_t,u_t)\,dW_t+c(t)\,dZ_t,\\
        X_{0-}=0.
        \end{array}
        \right.
    \end{equation}
where all parameters are measurable in their respective arguments and are such that the control problem makes sense; see, e.g. \cite{Haussmann-Suo-1995} for details.\footnote{Our specific assumptions on the model parameters are introduced in Section \ref{intro-MFG-singular} below.} The {\it regular control} $u=(u_t)_{t \in [0,T]}$ takes values in a compact metric space $U$, and the {\sl singular control} $Z=(Z_t)_{t \in [0,T]}$ takes values in $\mathbb{R}^d$. For convenience we sometimes write $Z\in\widetilde{\mathcal{A}}(\mathbb{R})$ by which we mean that the sample paths of the stochastic process $Z$ belong to $\widetilde{\mathcal{A}}(\mathbb{R})$. Similarly, we occasionally write $X\in\widetilde {\mathcal D}(\mathbb R)$ and $Y\in\widetilde {\mathcal C}(\mathbb R)$.
%
%

\subsubsection{Relaxed controls}

The existence of optimal {\sl relaxed controls} to stochastic singular control problems has been addressed in \cite{Haussmann-Suo-1995} using the so-called compactification method. We use a similar approach to solve MFGs with singular controls, albeit in different topological setting. The following notion of relaxed controls follows \cite{Haussmann-Suo-1995} where we adopt our convention that all processes are extended to the whole real line.

\begin{definition}\label{relaxed-control}
The tuple $r=(\Omega,\mathcal{F},\{\mathcal{F}_t,t\in\mathbb{R}\},\mathbb{P},X,\underline{Q},Z)$ is called a relaxed control if
    \begin{itemize}
        \item[1.] $(\Omega,\mathcal{F},\{\mathcal{F}_t,t\in\mathbb{R}\},\mathbb{P})$ is a filtered probability space;
        \item[2.] $\mathbb{P}(X_{t}=0,Z_{t}=0, \underline{Q}_t(du)=\delta_{\widetilde{u_0}}(du)\textrm{ if }t<0; X_{t}=X_T,Z_{t}=Z_T, \underline{Q}_t(du)=\delta_{\widetilde{u_T}}(du)\textrm{ if }t>T)=1$, for some $\widetilde{u_0},\widetilde{u_T}\in U$;
        \item[3.] $\underline{Q}:\mathbb{R}\times\Omega\rightarrow\mathcal{P}(U)$ is $\{\mathcal{F}_t,t\in\mathbb{R}\}$ progressively measurable,  $Z$ is $\{\mathcal{F}_t,t\in\mathbb{R}\}$ progressively measurable and $Z\in\widetilde{\mathcal{A}}(\mathbb{R})$;
        \item[4.] $X$ is a $\{\mathcal{F}_t,t\in\mathbb{R}\}$ adapted stochastic process, $X\in\widetilde{\mathcal{D}}(\mathbb{R})$ and for each $\phi\in \mathcal{C}^2_b(\mathbb{R}^d)$, the space of all continuous and bounded functions with continuous and bounded first- and second-order derivatives, $\mathcal{M}^{\phi}$ is a well defined $\mathbb{P}$ continuous martingale, where
            \begin{equation*}
            \begin{split}
            \mathcal{M}^{\phi}_t:=&~\phi(X_t)-\int_0^t\int_U\mathcal{L}\phi(s,X_s,u)\,\underline{Q}_s(du)ds-\int_0^t(\partial_x\phi(X_{s-}))^{\top}c(s)dZ_s\\
            &-\sum_{0\leq s\leq t}\big(\phi(X_s)-\phi(X_{s-})-(\partial_x\phi(X_{s-}))^{\top}\triangle X_s\big),\qquad t\in[0,T]
             \end{split}
             \end{equation*}
             with $\mathcal{L}\phi(t,x,u):=\frac{1}{2}\sum_{ij}a_{ij}(t,x,u)\frac{\partial^2\phi(x)}{\partial_{x_i}\partial_{x_j}}+\sum_ib_i(t,x,u)\partial_{x_i}\phi(x)$ and $a(t,x,u)=\sigma\sigma^{\top}(t,x,u)$.
    \end{itemize}
 \end{definition}
%

 The cost functional corresponding to a relaxed control $r$ is defined by
    \begin{equation}\label{cost-corresponding-relaxed-control}
        \widetilde{J}(r)=E^{\mathbb{P}}\left[\int_0^T\int_Uf(t,X_t,u)\,\underline{Q}_t(du)dt+\int_0^Th(t)\,dZ_t+g(X_T)\right].
    \end{equation}

Let $(\Omega,\mathcal{F},\{\mathcal{F}_t,t\in\mathbb{R}\},\mathbb{P},X,\underline{Q},Z)$ be a relaxed control. If the process $\underline Q$ is of the form $\underline Q_t(du)=\delta_{u_t}(du)$, for some progressively measurable $U$-valued process $u$, then we call $(\Omega,\mathcal{F},\{\mathcal{F}_t,t\in\mathbb{R}\},\mathbb{P},X,u,Z)$ a {\sl strict control}.\footnote{If there is no risk of confusion, then we call the processes $\underline Q$, respectively $u$ the relaxed, respectively strict control.} In particular, any strict control corresponds to a relaxed control. Relaxed control can thus be viewed as a form of mixed strategies over strict controls. In particular, both the cost function and the state dynamics (more precisely, the martingale problem) are linear in relaxed controls. Furthermore, compactness w.r.t.~relaxed controls is much easier to verify than compactness w.r.t.~strict controls. Under suitable convexity conditions on the model data, the optimization problem over the set of relaxed controls is equivalent to the one over strict controls as shown by the following remark.

\begin{remark}\label{relaxed-to-strict}
\begin{enumerate}
\item For $(t,x)\in[0,T]\times \mathbb{R}^d$, let
    \[
        K(t,x)=\{(a(t,x,u),b(t,x,u),e):~e\geq f(t,x,u),~u\in U\}.
    \]
If $K(t,x)$ is convex for each $(t,x)\in[0,T]\times \mathbb{R}^d$, then it can be shown that for each relaxed control, there exists a strict control and a singular control with smaller or equal cost. Indeed, by the proof of \cite[Theoerem 3.6]{haussmann-lepeltier-1990}, for any relaxed control $r=(\Omega,\mathcal{F},\mathcal{F}_t,\mathbb{P},X,\underline Q,Z)$, there exists a progressively measurable $U$-valued process $\bar{u}$ and a  $\mathbb{R}^+$-valued process $\bar{v}$ such that for almost all $(t, \omega) \in [0,T] \times \Omega$,
\begin{equation}\label{relaxed-control-to-strict-control}
\begin{split}
    &\left(\int_U a(t,X_t(\omega),u)\,\underline Q_t(\omega, du),\int_U b(t,X_t(\omega),u)\,\underline Q_t(\omega,du),\int_U f(t,X_t(\omega),u)\,\underline Q_t(\omega,du)\right)\\
    =&~\left(a(t,X_t(\omega),\bar u_t(\omega)),b(t,X_t(\omega),\bar u_t(\omega)),f(t,X_t(\omega), \bar u_t(\omega))+ \bar v_t(\omega)\right).
\end{split}
\end{equation}
Then $\alpha=(\Omega,\mathcal{F},\mathcal{F}_t,\mathbb{P},X,\bar{u},Z)$ is a strict control with smaller or equal cost.
\item If $a(t,x,u)=K_{t,x}u^2$, $b(t,x,u)=K_{t,x}'u^2$, $f(t,x,u)=K_{t,x}''u^2$, where $|K_{t,x}|,~|K_{t,x}'|,~|K_{t,x}''|\leq K$ for some positive constant $K$, then the set $K(t,x)$ is convex for each $(t,x)\in[0,T]\times\mathbb R^d$.
\end{enumerate}
\end{remark}

\subsubsection{Canonical state space and disintegration}\label{disintegration}
In what follows, we always assume that  $\Omega$ is the canonical path space, i.e.
\[
	\Omega=\widetilde{\mathcal{D}}(\mathbb{R})\times \widetilde{\mathcal{U}}(\mathbb R)\times \widetilde{\mathcal{A}}(\mathbb{R})
\]
and that the filtration $\{\mathcal{F}_t,t\in\mathbb{R}\}$ is generated by the coordinate projections $X,Q,Z$. More precisely, for each $\omega:=(x,q,z)\in\Omega$,
\[
    X(\omega)=x,\qquad Q(\omega)=q,\qquad Z(\omega)=z.
\]
%
%
and for $t\in[0,T]$, $\mathcal{F}_t:=\mathcal F^X_t\times \mathcal F^Q_t\times \mathcal F^Z_t$, where
\[
    \mathcal F^X_t=\sigma(X_s,s\leq t),\qquad \mathcal F^Q_t=\sigma(Q(S),S\in\mathcal{B}([0,t]\times U)),\qquad\mathcal F^Z_t=\sigma(Z_s,s\leq t);
\]
if $t<0$, then $\mathcal{F}_t:=\{\Omega,{\O}\}$ and if $t>T$, then $\mathcal{F}_t:=\mathcal{F}_T$.

The following argument shows that relaxed controls can be defined in terms of projection mappings. In fact, since $[0,T]$ and $U$ are compact, by the definition of $\widetilde{\mathcal U}(\mathbb R)$, each $q\in\widetilde{\mathcal U}(\mathbb R)$ allows for the disintegration $$q(dt,du)=q_t(du)dt$$
for some measurable $\mathcal P(U)$-valued function $q_t$. By \cite[Lemma 3.2]{Lacker-2015} and by definition of the space $\widetilde {\mathcal U}(\mathbb R)$ there exists a $\mathcal F^Q_t$-predictable $\mathcal P(U)$-valued process $\Pi$ such that for each $q\in\widetilde{\mathcal U}(\mathbb R)$, 
    \[
        \Pi_t(q)=q_t, ~a.e. ~t\in[0,T];\qquad \Pi_t(q)\equiv\delta_{\widetilde{u_0}},~t<0;\qquad \Pi_t(q)\equiv\delta_{\widetilde{u_T}},~t>T.
    \]
 Hence, the process $ Q^o_t:=\Pi_t\circ Q$
    is $\mathcal F_t$-predictable. As a result, for each $\omega=(x,q,z)$,
    \[
        Q(\omega)(dt,du)=q(dt,du)=q_t(du)dt=\Pi_t(q)(du)dt=\Pi_t\circ Q(\omega)(du)dt=Q^o_t(\omega)(du)dt.
    \]
This yields an adapted disintegration of $Q$ in terms of the $\{\mathcal{F}_t,t\in\mathbb{R}\}$ progressively measurable process
    \[
    	Q^ o:\mathbb{R}\times\Omega\rightarrow\mathcal{P}(U).
    \]
and hence allows us to define control rules. We notice that it is not appropriate to replace $\widetilde{\mathcal{U}}(\mathbb{R})$ in the definition of the canonical path space by the space of c\`adl\`ag $\mathcal{P}(U)$-valued functions as the definition of relaxed controls does not assume any path properties of $t \mapsto \underline{Q}_t.$
\begin{definition}\label{control-rule}
For the canonical path space $\Omega$, the canonical filtration $\{ \mathcal{F}_t, t \in \mathbb{R} \}$ and the coordinate projections $(X,Q,Z)$ introduced above, if
 $r=(\Omega,\mathcal{F},\{\mathcal{F}_t,t\in\mathbb{R}\},\mathbb{P},X,{Q}^o,Z)$ is a relaxed control in the sense of Definition \ref{relaxed-control}, then the probability measure $\mathbb{P}$ is called a {\sl control rule}. The associated cost functional is defined as
    \begin{equation*}\label{cost-with-control-rule}
        \widehat{J}(\mathbb{P}):=\widetilde{J}(r).
    \end{equation*}
\end{definition}
Let us denote by $\mathcal{R}$ the class of all the control rules for the stochastic control problem \eqref{model-stochastic-singular-control}. Clearly,
\[
	\inf\limits_{\mathbb{P}\in\mathcal{R}}\widehat{J}(\mathbb{P})\geq \inf\limits_{\textrm{relaxed control }r}\widetilde{J}(r).
\]
Conversely, for any relaxed control $r$ one can construct a control rule $\mathbb{P}\in\mathcal{R}$ such that $\widehat{J}(\mathbb{P})=\widetilde{J}(r).$ The proof is standard; it can be found in, e.g.  \cite[Proposition 2.6]{Haussmann-Suo-1995}. In other words, the optimization problems over relaxed controls and control rules are equivalent. It is hence enough to consider control rules. From now on, we let $(Q_t)_{t \in \mathbb{R}}:=(Q^o_t)_{t \in \mathbb{R}}$ for simplicity.

\begin{remark}
In \cite{Haussmann-Suo-1995} -  
with the choice of different topologies and
under suitable assumptions on the cost coefficients - it is shown that an optimal control rule exists if $g \equiv 0$. Their method allows for terminal costs only after a modification of the cost function; see \cite[Remark 2.2 and Section 4]{Haussmann-Suo-1995} for details. As a byproduct (see Corollary \ref{existence-singular-control}) of our analysis of MFGs, under the same assumptions on the coefficients as in \cite{Haussmann-Suo-1995} we establish the existence of an optimal control rule for terminal cost functions that satisfy a linear growth condition. In Section \ref{McKean-Vlasov-control} we furthermore outline a generalization of the stochastic singular control problem to problems of McKean-Vlasov-type. 
\end{remark}

\subsection{Mean field games with singular controls}\label{intro-MFG-singular}

We are now going to consider MFGs with singular controls of the form (\ref{model-MFGs-singular-control}). We again restrict ourselves to relaxed controls. Throughout the paper, for each $\mu\in\mathcal{P}_p(\widetilde{\mathcal{D}}(\mathbb{R}))$, put $\mu_t=\mu\circ\pi^{-1}_t$, where $\pi_t:x\in\widetilde{\mathcal{D}}(\mathbb{R})\rightarrow x_t$. The first step of solving mean field games is to solve the representative agent's optimal control problem
\begin{equation*}\label{model-stochastic-singular-control2}
        \left\{
        \begin{array}{ll}
	\inf_{u,Z} E\left[\int_0^Tf(t,X_t,\mu_t, u_t)\,dt+g(X_T,\mu_T)+\int_0^Th(t)\,dZ_t\right]\\
	\mbox{subject to}\\
	dX_t=b(t,X_t,\mu_t,u_t)\,dt+\sigma(t,X_t,\mu_t,u_t)\,dW_t+c(t)\,dZ_t,\\
        X_{0-}=0
        \end{array}
        \right.
    \end{equation*}
for any {\it fixed} mean field measure $\mu\in\mathcal{P}_p(\widetilde{\mathcal{D}}(\mathbb{R}))$. The canonical path space for MFGs with singular controls is
\[
	\Omega:=\widetilde{\mathcal{D}}(\mathbb{R})\times\widetilde{\mathcal{U}}(\mathbb R)\times\widetilde{\mathcal{A}}(\mathbb{R}).
\]	
We assume that the spaces $\widetilde{\mathcal{D}}(\mathbb{R})$ and $\widetilde{\mathcal{A}}(\mathbb{R})$ are endowed with the $M_1$ topology. We define a metric on the space ${\mathcal{U}}(\mathbb R)$ induced by the Wasserstein distance on compact time intervals by
 \begin{equation}\label{metric-on-mathcal-U}
 \begin{split}
    d_{{\mathcal{U}}(\mathbb R)}(q^1,q^2)&:=\mathcal{W}_{p,[0,T]\times U}\left(\frac{q^1}{T},\frac{q^2}{T}\right) \\
    & \quad +\sum_{n=0}^{\infty}  \frac{1}{2^{n+1}}\left\{\mathcal W_{p,[-(n+1),-n]\times U}(q^1,q^2)+
    \mathcal W_{p,[T+n,T+n+1]\times U}(q^1,q^2) \right\}.
    \end{split}
 \end{equation}
The space $\widetilde{\mathcal U}(\mathbb R)$ endowed with the metric $d_{\widetilde{\mathcal U}(\mathbb R)}:=d_{\mathcal U(\mathbb R)}$ is compact.
Furthermore, it is well known \cite[Chapter 3]{Whitt-2002} that the spaces $\widetilde{\mathcal{D}}(\mathbb{R})$ and $\widetilde{\mathcal{A}}(\mathbb{R})$ are Polish spaces when endowed with the $M_1$ topology, and that the $\sigma$-algebras on $\widetilde{\mathcal{D}}(\mathbb{R})$ and $\widetilde{\mathcal{A}}(\mathbb{R})$ coincide with the Kolmogorov $\sigma$-algebras generated by the coordinate projections.

\begin{definition}\label{control-rule-mu}
A probability measure $\mathbb{P}$ is called a control rule with respect to $\mu\in\mathcal{P}_p(\widetilde{\mathcal{D}}(\mathbb{R}))$ if
\begin{itemize}
        \item[1.] $(\Omega,\mathcal{F},\{\mathcal{F}_t,t\in\mathbb{R}\},\mathbb{P})$ is the canonical probability space and $(X,Q,Z)$ are the coordinate projections;
        \item[2.] 
            for each $\phi\in \mathcal{C}^2_b(\mathbb{R}^d)$, $\mathcal{M}^{\mu,\phi}$ is a well defined $\mathbb{P}$ continuous martingale, where
            \begin{equation}\label{martingale-problem-MFGs}
            \begin{split}
            \mathcal{M}^{\mu,\phi}_t:=&~\phi(X_t)-\int_0^t\int_U\mathcal{L}\phi(s,X_s,\mu_s,u)\,Q_s(du)ds-\int_0^t(\partial_x\phi(X_{s-}))^{\top}c(s)dZ_s\\
            &-\sum_{0\leq s\leq t}\big(\phi(X_s)-\phi(X_{s-})-(\partial_x\phi(X_{s-}))^{\top}\triangle X_s\big),\qquad t\in[0,T]
             \end{split}
             \end{equation}
             with $\mathcal{L}\phi(t,x,\nu,u):=\frac{1}{2}\sum_{ij}a_{ij}(t,x,\nu,u)\frac{\partial^2\phi(x)}{\partial_{x_i}\partial_{x_j}}+\sum_ib_i(t,x,\nu,u)\partial_{x_i}\phi(x)$ and $a(t,x,\nu,u)=\sigma\sigma^{\top}(t,x,\nu,u)$, for each $(t,x,\nu,u)\in[0,T]\times\mathbb{R}^d\times\mathcal{P}_p(\mathbb{R}^d)\times  U$.
    \end{itemize}
\end{definition}

For a fixed measure $\mu\in\mathcal{P}_p(\widetilde{\mathcal{D}}(\mathbb{R}))$, the corresponding set of control rules is denoted by $\mathcal{R}(\mu)$, the cost functional corresponding to a control rule $\mathbb{P}\in\mathcal{R}(\mu)$ is
    \[
        J(\mu,\mathbb{P})=E^{\mathbb{P}}\left[\int_0^T\int_Uf(t,X_t,\mu_t,u)\,Q_t(du)dt+\int_0^Th(t)\,dZ_t+g(X_T,\mu_T)\right],
    \]
and the (possibly empty) set of optimal control rules is denoted by
\[
	\mathcal{R}^*(\mu):=\textrm{argmin}_{\mathbb{P\in\mathcal{R}(\mu)}}J(\mu,\mathbb{P}).
\]
If a probability measure $\mathbb{P}$ satisfies the fixed point property
\[	
	\mathbb{P}\in \mathcal{R}^*(\mathbb{P}\circ X^{-1}),
\]	
then	we call $\mathbb{P}\circ X^{-1}$ or $\mathbb{P}$ or the associated tuple $(\Omega,\mathcal{F},\mathcal{F}_t,\mathbb{P},X,Q,Z)$ a {\sl relaxed solution} to the MFG with singular controls \eqref{model-MFGs-singular-control}. Moreover, if $\mathbb{P}\in \mathcal{R}^*(\mathbb{P}\circ X^{-1})$ and $\mathbb{P}(Q(dt,du)=\delta_{\bar{u}_t}(du)dt)=1$ for some progressively measurable process $\bar{u}$, then we call $\mathbb{P}\circ X^{-1}$ or $\mathbb{P}$ or the associated tuple $(\Omega,\mathcal{F},\mathcal{F}_t,\mathbb{P},X,\bar{u},Z)$ a {\sl strict solution}.

The following theorem gives sufficient conditions for the existence of a relaxed solution to our MFG. The proof is given in Section \ref{proof}.

\begin{theorem}\label{existence}
For some $\bar{p}>p\geq 1$, we assume that the following conditions are satisfied:
   \begin{itemize}
        \item[$\mathcal{A}_1$. ]There exists a positive constant $C_1$ such that $|b|\leq C_1$ and $|a|\leq C_1$; $b$ and $\sigma$ are measurable in $t\in[0,T]$ and continuous in $(x,\nu,u)\in\mathbb{R}^d\times\mathcal{P}_p(\mathbb{R}^d)\times U$; moreover, $b$ and $\sigma$ are Lipschitz continuous in $x\in\mathbb{R}^d$, uniformly in $(t,\nu,u)\in[0,T]\times\mathcal{P}_p(\mathbb{R}^d)\times U$.
        \item[$\mathcal{A}_2$. ] The functions $f$ and $g$ are measurable in $t\in[0,T]$ and are continuous with respect to $(x,\nu,u)\in\mathbb{R}^d\times\mathcal{P}_p(\mathbb{R}^d)\times U$.
        \item[$\mathcal{A}_3$. ] For each $(t,x,\nu,u)\in [0,T]\times\mathbb{R}^d\times\mathcal{P}_p(\mathbb{R}^d)\times U$, there exist strictly positive constants $C_2$, $C_3$ and a positive constant $C_4$ such that
            \[
                 -C_2\left(1-|x|^{\bar{p}}+\int_{\mathbb{R}^d}|x|^p\,\nu(dx)\right)\leq g(x,\nu)\leq C_3\left(1+|x|^{\bar{p}}+\int_{\mathbb{R}^d}|x|^p\,\nu(dx)\right),
            \]
            and
            \[
                 |f(t,x,\nu,u)|\leq C_4\left(1+|x|^p+|u|^p+\int_{\mathbb{R}^d}|x|^p\,\nu(dx)\right).
            \]
        \item[$\mathcal{A}_4$. ] The functions $c$ and $h$ are continuous and $c$ is strictly positive.
        \item[$\mathcal{A}_5$. ] The functions $b$, $\sigma$ and $f$ are locally Lipschitz continuous with $\mu$ uniformly in $(t,x,u)$, i.e., for $\varphi=b,~\sigma~\textrm{and}~f$, there exists $C_5>0$ such that for each $(t,x,u)\in[0,T]\times\mathbb R^d\times U$ and $\nu^1,\nu^2\in\mathcal P_p(\mathbb R^d)$ there holds that
                \[
                    |\varphi(t,x,\nu^1,u)-\varphi(t,x,\nu^2,u)|\leq C_5\Big(1+L(\mathcal{W}_p(\nu^1,\delta_0),\mathcal{W}_p(\nu^2,\delta_0))\Big) \mathcal{W}_p(\nu^1,\nu^2),
                \]
                where $L(\mathcal{W}_p(\nu^1,\delta_0),\mathcal{W}_p(\nu^2,\delta_0))$ is locally bounded with $\mathcal{W}_p(\nu^1,\delta_0)$ and $\mathcal{W}_p(\nu^2,\delta_0)$.
       \item[$\mathcal{A}_6$. ] $U$ is a compact metrizable space.
    \end{itemize}
Under assumptions $\mathcal{A}_1$-$\mathcal{A}_6$, there exists a relaxed solution to the MFGs with singular controls (\ref{model-MFGs-singular-control}).
\end{theorem}
\begin{remark}\label{comment-on-assumption}
A typical example where assumption $\mathcal{A}_3$ holds is
    $$g(x,\nu)=|x|^{\bar{p}}+\bar{g}(\nu),$$
where $|\bar{g}(\nu)|\leq \int_{\mathbb{R}^d}|y|^p\,\nu(dy)$. This assumption is not needed under a finite fuel constraint on the singular controls. It is needed in order to approximate MFGs with singular controls by MFGs with a finite fuel constraint. The assumption that $c>0$ is also only needed when passing from finite fuel constrained to unconstrained problems, see Lemma \ref{boundedness-expectation-pm-zt-pbar}. Assumption $\mathcal{A}_5$ is needed in order to prove the continuity of the cost function and the correspondence ${\cal R}$ in $\mu$. A typical example for $\mathcal{A}_5$ is $\int|x|^p\nu(dx)$ or $\int|x|^p\nu(dx)\wedge K$ for some fixed constant $K$ if  boundedness is required.
\end{remark}

\begin{remark}\label{existence-strict-MFG}  If we assume for each $(t,x,\nu)\in[0,T]\times\mathbb{R}^d\times\mathcal{P}_p(\mathbb{R}^d)$, $K(t,x,\nu)$ is convex, where
    \[
        K(t,x,\nu)=\{(a(t,x,\nu,u),b(t,x,\nu,u),e):~e\geq f(t,x,\nu,u),~u\in U\},
    \]
     a strict solution to our MFG can be constructed from a relaxed solution. Let $r^*=(\Omega,\mathcal{F},\mathcal{F}_t,\mathbb{P}^*,X,Q,Z)$ is a relaxed solution to MFG. Let $a^*(t,x,u)=a(t,x,\mu^*_t,u)$, $b^*(t,x,u)=b(t,x,\mu^*_t,u)$ and $f^*(t,x,u)=f(t,x,\mu^*_t,u)$, where $\mu^*=\mathbb{P}^*\circ X^{-1}$. Similar to Remark \ref{relaxed-to-strict}, there exist $U$-valued process $\bar{u}$ and $\mathbb{R}^+$-valued process $\bar{v}$ such that \eqref{relaxed-control-to-strict-control} holds with $a,b,f$ replaced by $a^*,b^*,f^*$, respectively. Define
     \[
        \alpha^*=(\Omega,\mathcal{F},\mathcal{F}_t,\mathbb{Q}^*,X,\bar{u},Z),
     \]
     where $\mathbb Q^*=\mathbb P^*\circ (X,\delta_{\bar u_t}(du)dt,Z)^{-1}$. 
     Then, $\alpha^*$ is a strict solution. The point is that the marginal distribution $\mu^*$ does not change when passing from $r^*$ to $\alpha^*$.
\end{remark}

\section{Proof of the main result}\label{proof}
The proof of Theorem \ref{existence} is split into two parts. In Section \ref{section-finite-fuel} we prove the existence of a solution to our MFG under a finite fuel constraint on the singular controls. The general case is established in Section \ref{Approximation} using an approximation argument.


\subsection{Existence under a finite fuel constraint}\label{section-finite-fuel}
In this section, we prove the existence of a relaxed solution to our MFG under a finite fuel constraint. That is, unless stated otherwise, we restrict the set of admissible singular controls to the set
\[
	\widetilde{\mathcal{A}}^m(\mathbb{R}):=\{z\in\widetilde{\mathcal{A}}(\mathbb{R}):z_T\leq m\}, 
\]
for some $m>0$. By Corollary \ref{compactness-finite-fuel}, the set $\widetilde{\mathcal{A}}^m(\mathbb{R})$ is $(\widetilde{\mathcal{D}}(\mathbb{R}),d_{M_1})$ compact.

We start with the following auxiliary result on the tightness of the distributions of the solutions to a certain class of SDEs. {The proof uses the definition of the distance $|x-[y,z]|$ of a point $x$ to a line segment} $[y,z]$ and the modified strong $M_1$ oscillation function $\widetilde{w}_s$ introduced in \eqref{strong-M1-oscillation-fun} and \eqref{modified-oscillation}, respectively.
\begin{proposition}\label{relative-compactness-general-result}
For each $n\in\mathbb{N}$, on a probability space $(\Omega^n,\mathcal{F}^n,\mathbb{P}^n)$, let $X^n$ satisfy the following SDE on $[0,T]$:
\begin{equation}
    dX_t^n={b}_n(t)\,dt+\,dM^n_t+d{c}_n(t),
\end{equation}
where the random coefficients ${b}_n$ is measurable and bounded uniformly in $n$, $M^n$ is a continuous martingale with uniformly bounded and absolutely continuous quadratic variation, and ${c}_n$ is monotone and c\`adl\`ag in time a.s.~and $\sup_n E^{\mathbb{P}^n}(|c_n(0)|\vee|c_n(T)|)^{\bar{p}}<\infty$. Moreover, assume that $X^n_t=0$ if $t<0$ and $X^n_t=X^n_T$ if $t>T$. Then,  the sequence $\{\mathbb{P}^n\circ (X^n)^{-1}\}_{n\geq 1}$ is relatively compact as a sequence in $\mathcal{W}_{p,(\widetilde{\mathcal D}(\mathbb{R}),d_{M_1})}$.
\end{proposition}
\begin{proof}
By the uniform boundedness of ${b}_n$, $E^{\mathbb{P}^n}(|c_n(0)|\vee|c_n(T)|)^{\bar{p}}$ and the quadratic variation of $M^n$, there exists a constant $C$ that is independent of $n$, such that
\begin{equation}\label{boundedness-expectation-any-order}
    E^{\mathbb{P}^n}\sup_{0\leq t\leq T}|X^n_t|^{\bar{p}}\leq C<\infty.
\end{equation}
By \cite[Definition 6.8(3)]{Villani-2009} it is thus sufficient to check the tightness of $\{\mathbb{P}^n\circ (X^n)^{-1}\}_{n\geq 1}$. This can be achieved by applying Proposition \ref{modified-tightness-criterion}. Indeed, the condition (\ref{tightness-criterion-uniform-boundedness}) holds, due to (\ref{boundedness-expectation-any-order}). Hence, one only needs to check that for each $\epsilon>0$ and $\eta>0$, there exists $\delta>0$ such that
    \[
        \sup_n\mathbb{P}^n(\widetilde{w}_s({X}^n,\delta)\geq \eta)<\epsilon.
    \]
To this end, we first notice that for each $t$ and $t_1,~t_2,~t_3$ satisfying $0\vee(t-\delta)\leq t_1<t_2<t_3\leq (t+\delta)\wedge T$, the monotonicity of $c_n$ implies
   \begin{equation}
   \begin{split}
      &|{X}^n_{t_2}-[{X}^n_{t_1},{X}^n_{t_3}]|\nonumber\\
                                   \leq &~ \left|\int_{t_1}^{t_2}{b}_n(s)\,ds+{M}^n_{t_2}-M^n_{t_1}\right|+\left|\int_{t_2}^{t_3}{b}_n(s)\,ds+{M}^n_{t_3}-M^n_{t_2}\right|\\
                                   &~+\inf_{{0}\leq\lambda\leq 1}\left|{c}_n({t_2})-\lambda c_n(t_1)-(1-\lambda)c_n(t_3)\right|\nonumber\\
                                   =& \left|\int_{t_1}^{t_2}{b}_n(s)\,ds+{M}^n_{t_2}-M^n_{t_1}\right|+\left|\int_{t_2}^{t_3}{b}_n(s)\,ds+{M}^n_{t_3}-M^n_{t_2}\right|.
   \end{split}
   \end{equation}
Similarly, for $t_1$ and $t_2$ satisfying $0\leq t_1<t_2\leq\delta$,
    \[
|X^n_{t_1}-[0,X^n_{t_2}]|\leq       \left|\int_{t_1}^{t_2}{b}_n(s)\,ds+{M}^n_{t_2}-M^n_{t_1}\right|.
    \]
Therefore,
  \begin{equation*}
      \widetilde{w}_s({X},\delta)
      \leq~ 3\sup_{ t}\sup_{ t_1, t_2}\left|\int_{t_1}^{t_2}{b}_n(s)\,ds+{M}^n_{t_2}-M^n_{t_1}\right|,
  \end{equation*}
where the first supremum extends over $0\leq t\leq T$ and the second one extends over $0\vee (t-\delta)\leq t_1\leq t_2\leq T\wedge (t+\delta)$.
By the Markov inequality and the boundedness of ${b}_n$ and the quadratic variation, this yields
\begin{equation}\label{result-from-Markov-inequality}
    \mathbb{P}^n(\widetilde{w}_s(X^n,\delta)\geq\eta)\leq\frac{k(\delta)}{\eta},
\end{equation}
for some positive function $k(\delta)$ that is independent of $n$ and $m$ with $\lim_{\delta\rightarrow 0}k(\delta)=0$.
\end{proof}

The next result shows that the class of all possible control rules is relatively compact. In a subsequent step this will allow us to apply Berge's maximum theorem.

\begin{lemma}\label{relative-compactness-union-control-rule}
Under assumptions $\mathcal{A}_1$, $\mathcal{A}_4$ and $\mathcal{A}_6$, the set $\bigcup_{\mu\in\mathcal{P}_p(\widetilde{\mathcal{D}}(\mathbb{R}))}\mathcal{R}(\mu)$ is relatively compact in $\mathcal{W}_{p}$.
\end{lemma}
\begin{proof}
Let $\{\mu^n\}_{n\geq 1}$ be any sequence in $\mathcal{P}_p(\widetilde{\mathcal{D}}(\mathbb{R}))$ and $\mathbb{P}^n\in \mathcal{R}(\mu^n), n\geq 1$. It is sufficient to show that $\{\mathbb{P}^n\circ X^{-1}\}_{ n\geq 1}$, $\{\mathbb{P}^n\circ Q^{-1}\}_{ n\geq 1}$ and $\{\mathbb{P}^n\circ Z^{-1}\}_{n\geq 1}$ are relatively compact. Since $U$ and $\widetilde{\mathcal{A}}^m(\mathbb{R})$ are compact by assumption and Corollary \ref{compactness-finite-fuel}, respectively, $\{\mathbb{P}^n\circ Q^{-1}\}_{ n\geq 1}$ and $\{\mathbb{P}^n\circ Z^{-1}\}_{ n\geq 1}$ are tight. Since $\widetilde{\mathcal{U}}(\mathbb R)$ and $\widetilde{\mathcal{A}}^m(\mathbb{R})$ are compact, these sequences are relatively compact in the topology induced by Wasserstein metric; see \cite[Definition 6.8(3)]{Villani-2009}.

It remains to prove the relative compactness of $\{\mathbb{P}^n\circ X^{-1}\}_{n\geq 1}$. Since $\mathbb{P}^n$ is a control rule associated with the measure $\mu^n$, for any $n$, it follows from Proposition \ref{result-from-Karoui} that there exist extensions $(\bar{\Omega},\bar{\mathcal{F}},\{\bar{\mathcal{F}}_t,t\in\mathbb R\},\mathbb{Q}^n)$ of the canonical path spaces and processes $({X}^n,{Q}^n,{Z}^n,M^n)$ defined on it, such that
    \[
        d{X}^n_t=\int_U b(t,{X}^n_t,\mu^n_t,u)\,{Q}^n_t(du)dt+\int_U\sigma(t,{X}^n_t,\mu^n_t,u)\,M^n(du,dt)+c(t)\,d{Z}^n_t
    \]
    and
    \[
        \mathbb{P}^n=\mathbb{P}^n\circ (X,Q,Z)^{-1}=\mathbb{Q}^n\circ({X}^n,{Q}^n,{Z}^n)^{-1},
    \]
where $M^n$ is a martingale measure on $(\bar{\Omega},\bar{\mathcal{F}},\{\bar{\mathcal{F}}_t\in\mathbb R\},\mathbb{Q}^n)$ with intensity ${Q}^n$.
Relative compactness of $\{\mathbb{P}^n\circ X^{-1}\}_{n\geq 1}$ now reduces to relative compactness of $\{\mathbb{Q}^n\circ(X^n)^{-1}\}_{n\geq 1}$, which is a direct consequence of the preceding Proposition \ref{relative-compactness-general-result}.
\end{proof}
\begin{remark}\label{remark-on-relative-compactness}
For the above result, the assumption $c>0$ is not necessary. To see this, we decompose ${X}^n$ as
   \begin{align*}
{X}^n_{\cdot}=\int_0^{\cdot}\int_U b(t,{X}^n_t,\mu^n_t,u)\,{Q}^n_t(du)dt+\int_0^{\cdot}\int_U\sigma(t,{X}^n_t,\mu^n_t,u)\,M^n(du,dt)+\int_0^{\cdot}c^+(t)\,d{Z}^n_t-\int_0^{\cdot}c^-(t)\,d{Z}^n_t,
   \end{align*}
where $c^+$ and $c^-$ are the positive and negative parts of $c$, respectively. By the boundedness of $b$ and $\sigma$, we see that
\[
    E^{\mathbb{Q}^n}|K^n_t-K^n_s|^4\leq C|t-s|^2,
\]
where
 $$K^n_{\cdot}:=\int_0^{\cdot}\int_U b(t,{X}^n_t,\mu^n_t,u)\,{Q}^n_t(du)dt+\int_0^{\cdot}\int_U\sigma(t,{X}^n_t,\mu^n_t,u)\,M^n(du,dt).$$

Kolmogorov weak tightness criterion implies that, for each $\epsilon>0$, there exists a compact set $\mathcal{K}_1\subseteq\widetilde{\mathcal{C}}(\mathbb{R})$ such that
    \[
        \inf_n\mathbb{Q}^n\left(K^n\in \mathcal{K}_1\right)\geq 1-\epsilon.
    \]
     Define
        $$\mathcal{K}_2:=\left\{\int_0^{\cdot}c^+(s)\,dz_s\in\widetilde{\mathcal{A}}(\mathbb{R}):~z_T\in\widetilde{\mathcal{A}}^m(\mathbb R)\right\}.$$
     and
         $$\mathcal{K}_3:=\left\{-\int_0^{\cdot}c^-(s)\,dz_s\in\widetilde{\mathcal{D}}(\mathbb{R}):z\in\widetilde{\mathcal{A}}^m(\mathbb{R})\right\}.$$
     Thus,
   \begin{equation}
    \begin{split}
       &~\inf_n\mathbb{P}^n\{\omega\in\Omega:X_{\cdot}(\omega)\in \mathcal{K}_1+\mathcal{K}_2+\mathcal{K}_3\}\\
      \geq &~\inf_n\mathbb{Q}^n\left\{\bar{\omega}\in\bar{\Omega}:K^n_{\cdot}\in \mathcal{K}_1,\int_0^{\cdot}c^+(s)\,d{Z}^n_s\in \mathcal{K}_2,-\int_0^{\cdot}c^-(s)\,d{Z}^n_s\in \mathcal{K}_3\right\}\\
        \geq & ~ 1-\epsilon.
    \end{split}
    \end{equation}
By Corollary \ref{compactness-finite-fuel}, $\mathcal{K}_2$ and $\mathcal{K}_3$ are $M_1$-compact subsets of $\widetilde{\mathcal{D}}(\mathbb{R})$. By Remark \ref{M1-stronger-L1}(1), $\mathcal{K}_1$ is also a $M_1$ compact subset. Note that  the elements of $\mathcal{K}_1$ do not jump and that $\int_0^{\cdot}c^+(s)\,dz_s$ and $\int_0^{\cdot}c^-(s)\,dz_s$ never jump at the same time. Thus, Proposition \ref{M1-continuity-addition} implies that $\mathcal{K}_1+\mathcal{K}_2+\mathcal{K}_3$ is a
 $M_1$-compact subset of $\widetilde{\mathcal{D}}(\mathbb{R})$.
\end{remark}
The next result states that the cost functional is continuous on the graph
\[
	\textrm{Gr}\mathcal{R} :=
	\{(\mu,\mathbb{P})\in\mathcal{P}_p(\widetilde{\mathcal{D}}(\mathbb{R}))\times\mathcal{P}_p(\Omega):~\mathbb{P}\in\mathcal{R}(\mu)\}.
\]
of the multi-function $\mathcal{R}$. This, too, will be needed to apply Berge's maximum theorem below.

\begin{lemma}\label{continuity-J-mu-P}
Suppose that $\mathcal{A}_1$-$\mathcal{A}_6$ hold. Then
$J:\textrm{Gr}\mathcal{R}\rightarrow\mathbb{R}$ is continuous.
\end{lemma}
\begin{proof}
For each $\mu\in\mathcal{P}_p(\widetilde{\mathcal{D}}(\mathbb{R}))$ and $\omega=(x,q,z)\in\Omega$, set
\begin{equation} \label{calJ}
	\mathcal{J}(\mu,\omega)=\int_0^T\int_Uf(t,x_t,\mu_t,u)\,q_t(du)dt+g(x_T,\mu_T)+\int_0^T h(t)\,dz_t.
\end{equation}
Thus $$J(\mu,\mathbb{P})=\int_{\Omega}\mathcal{J}(\mu,\omega)\mathbb{P}(d\omega).$$ In a first step we prove that $\mathcal{J}(\cdot,\cdot)$ is continuous in the first variable; in a second step we prove continuity and a polynomial growth condition in the second variable. The joint continuity of $J$ will be proved in the final step.

{\it Step 1: continuity in $\mu$.}
Let $\mu^n\rightarrow\mu$ in $\mathcal{W}_{p,(\widetilde{\mathcal{D}}(\mathbb{R}),d_{M_1})}$ and recall that $\mu^n_t=\mu^n\circ \pi_t^{-1}$ and $\mu_t=\mu\circ \pi_t^{-1}$, where $\pi$ is the projection on $\widetilde{\mathcal{D}}(\mathbb{R})$. We consider the first two terms on the r.h.s.~in (\ref{calJ}) separately, starting with the first one.
By assumption $\mathcal{A}_5$,
    \begin{equation}\label{f-continuity-with-mu}
    \begin{split}
    &\left|\int_0^T\int_Uf(t,x_t,\mu^n_t,u)\,q_t(du)dt-\int_0^T\int_Uf(t,x_t,\mu_t,u)\,q_t(du)dt\right|\\
    \leq& ~ C\int_0^T\left(1+L\left(\mathcal{W}_p(\mu^n_t,\delta_0),\mathcal{W}_p(\mu_t,\delta_0)\right)\right)\mathcal{W}_p(\mu^n_t,\mu_t)\,dt\\
    \leq&~C\left(\int_0^T\left(1+L\left(\mathcal{W}_p(\mu^n_t,\delta_0),\mathcal{W}_p(\mu_t,\delta_0)\right)\right)^{\frac{p}{p-1}}\,dt\right)^{1-\frac{1}{p}}\left(\int_0^T\mathcal{W}_p(\mu^n_t,\mu_t)^p\,dt\right)^{\frac{1}{p}}.
    \end{split}
    \end{equation}
The convergence $\mu^n\rightarrow\mu$ in $\mathcal{W}_{p,(\widetilde{\mathcal{D}}(\mathbb{R}),d_{M_1})}$ implies $\mu^n\rightarrow\mu$ weakly. By Skorokhod's representation theorem, there exists $\bar{X}^n$ and $\bar{X}$ defined on some probability space $(\mathbb{Q},\bar{\Omega},\bar{\mathcal{F}})$, such that
    \[
        \mu^n=\mathbb{Q}\circ(\bar{X}^n)^{-1}, \quad \mu=\mathbb{Q}\circ\bar{X}^{-1}
    \]
and
    \[
        d_{M_1}(\bar{X}^n,\bar{X})\rightarrow 0 \quad \mathbb{Q}\textrm{-}a.s.
    \]
Hence, (\ref{f-continuity-with-mu}) implies that
\begin{equation*}
    \begin{split}
    &\left|\int_0^T\int_Uf(t,x_t,\mu^n_t,u)\,q_t(du)dt-\int_0^T\int_Uf(t,x_t,\mu_t,u)\,q_t(du)dt\right|\\
    \leq& ~ C\left(\int_0^T\left(1+L\left(\mathcal{W}_p(\mathbb{Q}\circ(\bar{X}^n_t)^{-1},\delta_0),\mathcal{W}_p(\mathbb{Q}\circ \bar{X}^{-1}_t,\delta_0)\right)\right)^{\frac{p}{p-1}}\,dt\right)^{1-\frac{1}{p}}\left(E^{\mathbb{Q}}\int_0^T|\bar{X}^n_t-\bar{X}_t|^p\,dt\right)^{\frac{1}{p}}
    \end{split}
    \end{equation*}
By Remark \ref{M1-stronger-L1}, we have
\[
    \int_0^T|\bar{X}^n_t-\bar{X}_t|^p\,dt\rightarrow 0 \qquad\textrm{ a.s. }\mathbb{Q}.
\]
Moreover, we have
\[
	\int_0^T|\bar{X}^n_t-\bar{X}_t|^p\,dt
	\leq 2^pT\left(d_{M_1}(\bar{X}^n,0)^p+d_{M_1}(\bar{X},0)^p\right).
\]	
	
	On the other hand,
    \begin{equation*}
    \begin{split}
        E^{\mathbb{Q}}\left(d_{M_1}(\bar{X}^n,0)^p+d_{M_1}(\bar{X},0)^p\right)=&~\int_{\mathcal{D}[0,T]}d_{M_1}(x,0)^p\,\mu^n(dx)+\int_{\mathcal{D}[0,T]}d_{M_1}(x,0)^p\,\mu(dx)\\
        \rightarrow&~2\int_{\mathcal{D}[0,T]}d_{M_1}(x,0)^p\,\mu(dx)<\infty.
    \end{split}
    \end{equation*}
Therefore, dominated convergence yields
\begin{equation}\label{convergence-from-skorokhod}
E^{\mathbb{Q}}\int_0^T|\bar{X}^n_t-\bar{X}_t|^p\,dt\rightarrow 0.
\end{equation}
Since
$\sup_n\mathcal{W}_p(\mathbb{Q}\circ(\bar{X}^n_t)^{-1},\delta_0)<\infty$
it thus follows from the local boundedness of the function $L$ that
    \begin{equation}\label{uniform-continuity-f-mu}
        \left|\int_0^T\int_Uf(t,x_t,\mu^n_t,u)\,q_t(du)dt-\int_0^T\int_Uf(t,x_t,\mu_t,u)\,q_t(du)dt\right|\rightarrow 0,\qquad \textrm{uniformly in }\omega.
    \end{equation}

As for the second term on the r.h.s.~in \eqref{calJ} recall first that $x^n\rightarrow x$ in $M_1$ implies $x^n_t\rightarrow x_t$ for each $t \notin Disc(x)$ and $x^n_T\rightarrow x_T$. In particular, the mapping $x \mapsto \varphi(x_T)$ is continuous for any continuous real-valued function $\varphi$ on $\mathbb{R}^d$. Since any continuous positive function $\varphi$ on $\mathbb{R}^d$ that satisfies $\varphi(x)\leq C(1+|x|^p)$, also satisfies
\[
	\varphi(x_T)\leq C(1+|x_T|^p)\leq C(1+d_{M_1}(x,0)^p)
\]	
we see that
    \begin{equation*}
        \begin{split}
         \left|\int_{\mathbb{R}^d}\varphi(x)\,\mu^n_T(dx)-\int_{\mathbb{R}^d}\varphi(x)\,\mu_T(dx)\right| =\left|\int_{\widetilde{\mathcal{D}}(\mathbb{R})}\varphi(x_T)\,\mu^n(dx)-\int_{\widetilde{\mathcal{D}}(\mathbb{R})}\varphi(x_T)\,\mu(dx)\right|
         \stackrel{n \to \infty}{\longrightarrow} 0.
       \end{split}
    \end{equation*}
More generally, we obtain $\mu^n_T\rightarrow\mu_T$ from $\mu^n\rightarrow\mu$, which also implies that $g(x_T,\mu^n_T)\rightarrow g(x_T,\mu_T)$.

{\it Step 2: continuity in $\omega$}. If $\omega^n=(x^n,q^n,z^n)\rightarrow\omega=(x,q,z)$, then $x^n_T\rightarrow x_T$. In particular, $$g(x^n_T,\mu_T)\rightarrow g(x_T,\mu_T).$$ Moreover, $z^n\rightarrow z$ in $M_1$ implies $z^n_t\rightarrow z_t$ for for all continuity points of $z$ and $z^n_T\rightarrow z_T$. By the Portmanteau theorem this implies that $$\int_0^T h(t)\,dz^n_t\rightarrow\int_0^Th(t)\,dz_t.$$

Next we show the convergence of $\int_0^T\int_U f(t,x^n_t,\mu_t,u)\,q^n_t(du)dt$ to $\int_0^T\int_U f(t,x_t,\mu_t,u)\,q_t(du)dt$. 
By Assumption ${\cal A}_2$ the convergence of $x^n$ to $x$ yields $f(t,x^n_t,\mu_t,u)\rightarrow f(t,x_t,\mu_t,u)$ for each $t \notin Disc(x)$. From the compactness of $U$ it follows that $\sup_{u\in U}\left|f(t,x^n_t,\mu_t,u)-f(t,x_t,\mu_t,u)\right|\rightarrow 0$ for each $t \notin Disc(x)$. Since $Disc(x)$ is at most countable this implies
    \begin{equation*}
    \begin{split}
        & \left|\int_0^T\int_U f(t,x^n_t,\mu_t,u)\,q^n_t(du)dt-\int_0^T\int_U f(t,x_t,\mu_t,u)\,q^n_t(du)dt\right| \\
        \leq & \int_0^T \sup_{u\in U} \left| f(t,x^n_t,\mu_t,u)-f(t,x_t,\mu_t,u) \right| dt  \rightarrow 0.
        \end{split}
    \end{equation*}
By \cite[Definition 6.8]{Villani-2009}, $q^n\rightarrow q$ in $d_{\widetilde{\mathcal{U}}(\mathbb R)}$ implies $q^n\rightarrow q$ weakly. Moreover, the first marginal of $q^n$ is Lebesgue measure. Thus, by \cite[Corollary 2.9]{JM-1981}, $q^n$ converges to $q$ in the stable topology (cf.~\cite[Definition 1.2]{JM-1981} for the definition of the stable topology). For fixed $(x,\mu)\in\widetilde{\mathcal D}(\mathbb R)\times\mathcal{P}_p(\widetilde{\mathcal D}(\mathbb R))$, the compactness of $U$ and the growth condition on $f$ implies the boundedness of $f$. Hence the definition of stable topology yields that
\[
	\lim_{n\rightarrow\infty}\left|\int_0^T\int_U f(t,x_t,\mu_t,u)\,q^n_t(du)dt-\int_0^T\int_U f(t,x_t,\mu_t,u)\,q_t(du)dt\right|=0.
\]
So we get the convergence
\[
	\lim_{n\rightarrow\infty}\left|\int_0^T\int_U f(t,x^n_t,\mu_t,u)\,q^n_t(du)dt-\int_0^T\int_U f(t,x_t,\mu_t,u)\,q_t(du)dt\right|=0.
\]

{\it Step 3: joint continuity of $J$}. Thus far, we have established the {\sl separate} continuity of the mapping $(\mu,\omega)\rightarrow\mathcal{J}(\mu,\omega)$. We are now going to apply \cite[Definition 6.8(4)]{Villani-2009} to prove the {\sl joint} continuity of $J$.

To this end, notice first that for each fixed $\mu\in\mathcal{P}_p(\widetilde{\mathcal{D}}(\mathbb{R}))$, due to Assumption $\mathcal{A}_3$,
    \begin{equation*}
      \begin{split}
         &\left|\int_0^T\int_U f(t,x_t,\mu_t,u)\,q_t(du)dt+\int_0^T h(t)\,dz_t\right|\\
          \leq &~ C\left(1+\int_0^T\int_U\left(1+|x_t|^p+|u|^p+\int_{\mathbb{R}^d}|y|^p\mu_t(dy)\right)\,q_t(du)dt+z_T\right)\\
          \leq &~ C\left(1+d_{M_1}(x,0)^p+\mathcal{W}_{p,[0,T]\times U}\left(\frac{q}{T},\delta_0\right)^p+d_{M_1}(z,0)+\int_0^T\int_{\mathbb{R}^d}|y|^p\,\mu_t(dy)dt\right)\\
          \leq&~ C\left(1+d_{M_1}(x,0)^p+\mathcal{W}_{p,[0,T]\times U}\left(\frac{q}{T},\delta_0\right)^p+d_{M_1}(z,0)^p+\int_{\widetilde{\mathcal{D}}(\mathbb{R})}d_{M_1}(y,0)^p\,\mu(dy)\right).
      \end{split}
    \end{equation*}
    Hence, using the uniform convergence \eqref{uniform-continuity-f-mu}, it follows from \cite[Definition 6.8]{Villani-2009} that $(\mu^n,\mathbb{P}^n)\rightarrow(\mu,\mathbb{P})$ implies that
{\small        \begin{equation}\label{continuity-J-mu-P-(1)}
    \begin{split}
        & \left|E^{\mathbb{P}^n}\left(\int_0^T\int_U f(t,X_t,\mu_t^n,u)\,Q_t(du)dt+\int_0^T h(t)\,dZ_t\right)-E^{\mathbb{P}}\left(\int_0^T\int_U f(t,X_t,\mu_t,u)\,Q_t(du)dt+\int_0^Th(t)\,dZ_t\right)\right| \\
        \leq& \left|E^{\mathbb{P}^n}\left(\int_0^T\int_U f(t,X_t,\mu^n_t,u)\,Q_t(du)dt+\int_0^T h(t)\,dZ_t\right)-E^{\mathbb{P}^n}\left(\int_0^T\int_U f(t,X_t,\mu_t,u)\,Q_t(du)dt+\int_0^T h(t)\,dZ_t\right)\right|\\
&+\left|E^{\mathbb{P}^n}\left(\int_0^T\int_U f(t,X_t,\mu_t,u)\,Q_t(du)dt+\int_0^T h(t)\,dZ_t\right)-E^{\mathbb{P}}\left(\int_0^T\int_U f(t,X_t,\mu_t,u)\,Q_t(du)dt+\int_0^T h(t)\,dZ_t\right)\right|\\
\rightarrow &0.
        \end{split}
    \end{equation}}
Since the terminal cost functions is not necessarily Lipschitz continuous we need to argue differently in order to prove the continuous dependence of the expected terminal cost on $(\mu,\mathbb{P})$. First, we notice that for
each $\widetilde{p}>\bar{p}$, by the boundedness of $b$, $\sigma$ and $Z$, we have that
    \begin{equation}\label{uniform-bounded-any-order}
        \sup_nE^{\mathbb{P}^n}d_{M_1}(X,0)^{\widetilde{p}} \leq C<\infty,
    \end{equation}
which implies
    \begin{equation}\label{uniform-tightness-condition}
        \lim_{K\rightarrow\infty}\sup_n\int_{\{x:d_{M_1}(x,0)>K\}} d_{M_1}(x,0)^{\bar{p}}\,\mathbb{P}^n(dx)=0.
    \end{equation}
By Assumption ${\cal A}_3$,
\[
	|g(x_T,\mu_T)|\leq C\left(1+|x_T|^{\bar{p}}+\int |y|^p\mu_T(dy)\right)\leq C\left(1+|x_T|^{\bar{p}}\right).
\]	
Together with (\ref{uniform-tightness-condition}) this implies,
    \begin{equation}\label{continuity-J-mu-P-(2)}
        E^{\mathbb{P}^n}g(X_T,\mu_T)\rightarrow E^{\mathbb{P}}g(X_T,\mu_T).
    \end{equation}
By the tightness of $\{\mathbb{P}^n\}_{n\geq 1}$, for each $\epsilon>0$, there exists a compact set $K_{\epsilon}\subseteq\widetilde{\mathcal{D}}(\mathbb{R})$ such that
\begin{equation}\label{continuity-J-mu-P-(3)}
\begin{split} &~\left|\int_{\widetilde{\mathcal{D}}(\mathbb{R})}g(x_T,\mu^n_T)\mathbb{P}^n(dx)-\int_{\widetilde{\mathcal{D}}(\mathbb{R})}g(x_T,\mu_T)\mathbb{P}^n(dx)\right|\\
\leq&~\int_{K_{\epsilon}}|g(x_T,\mu^n_T)-g(x_T,\mu_T)|\mathbb{P}^n(dx)+\int_{\widetilde{\mathcal{D}}(\mathbb{R})/K_{\epsilon}}|g(x_T,\mu^n_T)-g(x_T,\mu_T)|\mathbb{P}^n(dx)\\
\leq&~\sup_{x\in K_{\epsilon}}|g(x_T,\mu^n_T)-g(x_T,\mu_T)|+\left(\int_{\widetilde{\mathcal{D}}(\mathbb{R})/K_{\epsilon}}|g(x_T,\mu^n_T)-g(x_T,\mu_T)|^2\mathbb{P}^n(dx)\right)^{\frac{1}{2}}\left(\sup_n\mathbb{P}^n(\widetilde{\mathcal{D}}(\mathbb{R})/K_{\epsilon})\right)^{\frac{1}{2}}\\
\leq&~\sup_{x\in K_{\epsilon}}|g(x_T,\mu^n_T)-g(x_T,\mu_T)|+C\epsilon^{\frac{1}{2}}~~~(\textrm{by }(\ref{uniform-bounded-any-order})).
\end{split}
\end{equation}
Thus,
\begin{equation}\label{continuity-J-mu-P-(4)}
\left|\int_{\widetilde{\mathcal{D}}(\mathbb{R})}g(x_T,\mu^n_T)\mathbb{P}^n(dx)-\int_{\widetilde{\mathcal{D}}(\mathbb{R})}g(x_T,\mu_T)\mathbb{P}^n(dx)\right|\rightarrow 0.
\end{equation}
The convergence (\ref{continuity-J-mu-P-(1)}), (\ref{continuity-J-mu-P-(2)}) and (\ref{continuity-J-mu-P-(4)}) yield the joint continuity of $J(\cdot,\cdot)$.
\end{proof}

\begin{remark}\label{joint-LSC-J}
The preceding lemma shows that under a finite fuel constraint the cost functional $J$ is jointly continuous. In general, $J$ is only lower semi-continuous. In fact, for each positive constant $K$, let $g_K(\cdot) := g(\cdot) \wedge K$ and
\begin{equation*}
\begin{split}
    \mathcal J_K(\mu,\omega)
    :=&~\int_0^T\int_Uf(t,x_t,\mu_t,u)\,q_t(du)dt+g_K(x_T,\mu_T)+\int_0^T h(t)\,dz_t
\end{split}
\end{equation*}
By assumption $\mathcal A_3$, we have
\[
    |g_K(x,\mu)|\leq2K+C_2\left(1+\int_{\mathbb R^d}|y|^p\mu(dy)\right)\leq C\left(1+\int_{\mathbb R^d}|y|^p\mu(dy)\right).
\]
So \eqref{continuity-J-mu-P-(2)} and \eqref{continuity-J-mu-P-(3)} still hold with $g$ replaced by $g_K$ while \eqref{continuity-J-mu-P-(1)} still holds for $f$ and $h$. So $(\mu^n,\mathbb P^n)\rightarrow (\mu,P)$ implies
\[
    \int_\Omega\mathcal J_K(\mu^n,\omega) \mathbb{P}^n(d\omega)\rightarrow    \int_\Omega\mathcal J_K(\mu,\omega) \mathbb{P}(d\omega).
\]
Thus, by monotone convergence theorem, we have
\[
    \liminf_{n\rightarrow\infty}\int_\Omega\mathcal J(\mu^n,\omega) \mathbb{P}^n(d\omega)\geq \int_\Omega\mathcal J(\mu,\omega) \mathbb{P}(d\omega).
\]
\end{remark}

We now recall from \cite[Proposition 3.1]{Haussmann-Suo-1995} an equivalent characterization for the set of control rules $\mathcal{R}(\mathcal{\mu})$. This equivalent characterization allows us to verify the martingale property of the state process by verifying the martingale property of its continuous part. Since it is difficult to locate the proof, we give a sketch one in Appendix \ref{sketch-proof}.
\begin{proposition}\label{transfer-X-to-Y}
A probability measure $\mathbb{P}$ is a control rule with respect to the given $\mu\in\mathcal{P}_p(\widetilde{\mathcal{D}}(\mathbb{R}))$ if and only if there exists an $\mathcal{F}_t$ adapted process $Y\in{\mathcal{C}}(0,T)$ on the filtered canonical space $(\Omega,\mathcal{F},\mathcal{F}_t)$ such that
    \begin{itemize}
        \item[(1)] $\mathbb{P}(\omega\in\Omega:~X_{t}(\omega)=Y_{t}(\omega)+\int_0^{t}c(s)\,dZ_{s}(\omega),t\in[0,T])=1$;
        \item[(2)] for each $\phi\in\mathcal{C}^2_b(\mathbb{R}^d)$, $\overline{\mathcal{M}}^{\mu,\phi}$ is a continuous $\left(\mathbb{P},\mathcal{F}_t\right)$ martingale, where
             \begin{equation}
             \overline{\mathcal{M}}^{\mu,\phi}_t=\phi(Y_t)-\int_0^t\int_U \bar{\mathcal{L}}\phi(s,X_s,Y_s,\mu_s,u)\,Q_s(du)ds,\qquad t\in[0,T]
              \end{equation}
              with $\bar{\mathcal{L}}\phi(s,x,y,\nu,u)=\sum_i b_i(s,x,\nu,u)\partial_{y_i}\phi(y)+\frac{1}{2}\sum_{ij}a_{ij}(s,x,\nu,u)\frac{\partial^2\phi(y)}{\partial_{y_i}\partial_{y_j}}$ for each $(t,x,y,\nu,u)\in[0,T]\times\mathbb{R}^d\times\mathbb{R}^d\times\mathcal{P}_p(\mathbb{R}^d)\times  U$.
    \end{itemize}
\end{proposition}

The previous characterization of control rules allows us to show that the correspondence $\mathcal{R}$ has a closed graph.

\begin{proposition}\label{admissibility-limit-P-proposition}
Suppose that $\mathcal{A}_1$ and $\mathcal{A}_4$-$\mathcal{A}_6$ hold.
For any sequence $\{\mu^n\}_{n\geq 1}\subseteq\mathcal{P}_p(\widetilde{\mathcal{D}}(\mathbb{R}))$ and $\mu\in\mathcal{P}_p(\widetilde{\mathcal{D}}(\mathbb{R}))$ with $\mu^n\rightarrow\mu$ in $\mathcal{W}_{p,(\widetilde{\mathcal{D}}(\mathbb{R}),d_{M_1})}$, if $\mathbb{P}^n\in\mathcal{R}(\mu^n)$ and $\mathbb{P}^n\rightarrow\mathbb{P}$ in $\mathcal{W}_{p}$, then $\mathbb{P}\in\mathcal{R}(\mu)$.
\end{proposition}
\begin{proof}

In order to verify conditions (1) and (2), notice first that, for each $n$, there exists a stochastic process $Y^n\in{\mathcal{C}}(0,T)$ such that $$\mathbb{P}^n\left(X_t=Y^n_t+\int_0^t c(s)\,dZ_s, t\in[0,T] \right)=1$$ and such that the corresponding martingale problem is satisfied. In order to show that a similar decomposition and the martingale problem hold under the measure $\mathbb P$ we apply Proposition \ref{result-from-Karoui}. For each $n$, there exits a probability space $(\Omega^n,\mathcal{F}^n,\mathbb{Q}^n)$ that supports random variables $(\bar{X}^n,\bar{Q}^n,\bar{Z}^n)$ and a martingale measure $M^n$ with intensity $\bar{Q}^n$ such that $$\mathbb{P}^n=\mathbb{Q}^n\circ (\bar{X}^n,\bar{Q}^n,\bar{Z}^n)^{-1}$$ and
    \[
        d\bar{X}^n_t=\int_U b(t,\bar{X}^n_t,\mu^n_t,u)\,\bar{Q}^n_s(du)ds+\int_U \sigma(t,\bar{X}^n_t,\mu^n_t,u)\,M^n(du,dt)+c(t)d\bar{Z}^n_t.
    \]
Thus, for each $0\leq s<t\leq T$,
    \begin{equation}
        \begin{split}
            E^{\mathbb{P}^n}|Y^n_t-Y^n_s|^4=&~E^{\mathbb{P}^n}\left|\left(X_t-\int_0^tc(r)\,dZ_r\right)-\left(X_s-\int_0^sc(r)\,dZ_r\right)\right|^4\\
            =&~E^{\mathbb{Q}^n}\left|\left(\bar{X}^n_t-\int_0^tc(r)\,d\bar{Z}^n_r\right)-\left(\bar{X}^n_s-\int_0^sc(r)\,d\bar{Z}^n_r\right)\right|^4\\
            =&~E^{\mathbb{Q}^n}\left|\int_s^t\int_U b(r,\bar{X}^n_r,\mu^n_r,u)\,\bar{Q}^n_r(du)dr+\int_s^t\int_U \sigma(r,\bar{X}^n_r,\mu^n_r,u)\,M^n(du,dr)\right|^4\\
            \leq&~ C|t-s|^2.
        \end{split}
    \end{equation}
Hence, Kolmogorov's weak compactness criterion implies the tightness of $Y^n$. Therefore, taking a subsequence if necessary, the sequence $(X,Q,Z,Y^n)$ of random variables taking values in ${\Omega} \times {\cal C}(0,T)$ has weak limit $(\widehat{X},\widehat{Q},\widehat{Z},\widehat{Y})$ defined on some probability space.

By Skorokhod's representation theorem, there exists a probability space $(\widetilde{\Omega},\widetilde{\mathcal{F}},\mathbb{Q})$ that supports random variables $(\widetilde{X}^n,\widetilde{Q}^n,\widetilde{Z}^n,\widetilde{Y}^n)$ and $(\widetilde{X},\widetilde{Q},\widetilde{Z},\widetilde{Y})$ such that
\[
	{Law}(\widetilde{X}^n,\widetilde{Q}^n,\widetilde{Z}^n,\widetilde{Y}^n)={Law}(X,Q,Z,Y^n),
	\quad {Law}(\widetilde{X},\widetilde{Q},\widetilde{Z},\widetilde{Y})={Law}(\widehat{X},\widehat{Q},\widehat{Z},\widehat{Y})
\]
 and
\[
	(\widetilde{X}^n,\widetilde{Q}^n,\widetilde{Z}^n,\widetilde{Y}^n)\rightarrow(\widetilde{X},\widetilde{Q},\widetilde{Z},\widetilde{Y}) \quad \mathbb{Q}\mbox{-a.s.}
\]	
In particular, $\tilde{Y}\in\mathcal{C}(0,T)$ as the uniform limit of a sequence of continuous processes, and
\[
	\mathbb{Q} \left(\widetilde{X}_{t}=\widetilde{Y}_{t} +\int_0^{t}c(s)\,d\widetilde{Z}_s, t\in[0,T]\right)=1.
\]	
Since $\mathbb{P}^n\rightarrow \mathbb{P}$, we have $\mathbb{P}\circ(X,Q,Z)^{-1}=\mathbb{Q}\circ(\widetilde{X},\widetilde{Q},\widetilde{Z})^{-1}$. Hence, there exists a stochastic process $Y\in\mathcal{C}(0,T)$ such that
    $$\mathbb{P} \left( X_{t}=Y_{t}+\int_0^{t}c(s)\,dZ_s, t\in[0,T] \right)=1$$
 and $\mathbb{P}\circ(X,Q,Z,Y)^{-1}=\mathbb{Q}\circ(\widetilde{X},\widetilde{Q},\widetilde{Z},\widetilde{Y})^{-1}$.
Finally, for each $t\in[0,T]$, define
    \[
        \overline{\mathcal{M}}^{n,\mu^n,\phi}_t=\phi(Y^n_t)-\int_0^t\int_U\bar{\mathcal{L}}(s,X_s,Y^n_s,\mu^n_s,u)\,Q_s(du)ds,
    \]
    \[
         \widetilde{\mathcal{M}}^{n,\mu^n,\phi}_t=\phi(\widetilde{Y}^n_t)-\int_0^t\int_U\bar{\mathcal{L}}(s,\widetilde{X}^n_s,\widetilde{Y}^n_s,\mu^n_s,u)\,\widetilde{Q}^n_s(du)ds,
    \]
and
    \[
        \widetilde{\mathcal{M}}^{\mu,\phi}_t=\phi(\widetilde{Y}_t)-\int_0^t\int_U\bar{\mathcal{L}}(s,\widetilde{X}_s,\widetilde{Y}_s,\mu_s,u)\,\widetilde{Q}_s(du)ds.
    \]
For each $0\leq s<t\leq T$ and each $F$ that is continuous, bounded and $\mathcal{F}_s$-measurable, we have
    \begin{equation}
        \begin{split}
            0=&~E^{\mathbb{P}^n}\left(\overline{\mathcal{M}}^{n,\mu^n,\phi}_t-\overline{\mathcal{M}}^{n,\mu^n,\phi}_s\right)F(X,Q,Z)=E^{\mathbb{Q}}\left(\widetilde{\mathcal{M}}^{n,\mu^n,\phi}_t-\widetilde{\mathcal{M}}^{n,\mu^n,\phi}_s\right)F(\widetilde{X}^n,\widetilde{Q}^n,\widetilde{Z}^n)\\
            &\rightarrow E^{\mathbb{Q}}\left(\widetilde{\mathcal{M}}^{\mu^*,\phi}_t-\widetilde{\mathcal{M}}^{\mu^*,\phi}_s\right)F(\widetilde{X},\widetilde{Q},\widetilde{Z})=E^{\mathbb{P}}\left(\overline{\mathcal{M}}^{\mu,\phi}_t-\overline{\mathcal{M}}^{\mu,\phi}_s\right)F(X,Q,Z).
        \end{split}
    \end{equation}
\end{proof}
\begin{remark}\label{remark-on-admissibility}
Note that the proof of Proposition \ref{admissibility-limit-P-proposition}, does not require the finite fuel constraint.
\end{remark}
The next corollary shows that the correspondence $\mathcal R$ is continuous in the sense of \cite[Definition 17.2, Theorem 17.20, 17.21]{AB-2006}.
%
\begin{corollary}\label{continuity-R}
Suppose that $\mathcal{A}_1$, $\mathcal{A}_4$-$\mathcal{A}_6$ hold. Then,  $\mathcal{R}:\mathcal{P}_p(\widetilde{\mathcal{D}}(\mathbb{R}))\rightarrow 2^{\mathcal{P}_p(\Omega)}$ is continuous and compact-valued.
\end{corollary}
\begin{proof}
The lower hemi-continuity of $\mathcal{R}$ can be dealt with as \cite[Lemma 4.4]{Lacker-2015} since $b$ and $\sigma$ are Lipschitz continuous in $x$. Lemma \ref{relative-compactness-union-control-rule}, Proposition \ref{admissibility-limit-P-proposition} and \cite[Theorem 17.20]{AB-2006} imply that $\mathcal R$ is upper hemi-continuous and compact-valued.
\end{proof}


\begin{corollary}\label{existence-singular-control}
Under assumptions $\mathcal{A}_1$-$\mathcal{A}_6$, $\mathcal{R}^*(\mu)\neq{\O}$ for each $\mu\in\mathcal{P}_p(\widetilde{\mathcal{D}}(\mathbb{R}))$ and $\mathcal{R}^*$ is upper hemi-continuous.
\end{corollary}
\begin{proof}
By  \cite[Section 5.4]{Karatzas-Shreve-1991}, for each $\mu\in\mathcal{P}_p(\widetilde{\mathcal{D}}(\mathbb{R}))$ the set $\mathcal{R}(\mu)$ is nonempty. Corollary \ref{continuity-R} implies that $\mathcal{R}$ is compact-valued and continuous. By Lemma \ref{continuity-J-mu-P}, $J:Gr\mathcal{R}\rightarrow\mathbb{R}$ is jointly continuous. Thus, \cite[Theorem 17.31]{AB-2006} yields that $\mathcal{R}^*$ is nonempty valued and upper hemi-continuous.
\end{proof}
\begin{remark}\label{existence-singular-control-remark}
Corollary \ref{existence-singular-control} in fact shows that the stochastic singular control problem (\ref{model-stochastic-singular-control}) admits an optimal control rule in the sense of Definition \ref{control-rule}. Using our method, we could have obtained Corollary \ref{existence-singular-control} under the same assumptions of the coefficients as in \cite{Haussmann-Suo-1995}. We will generalize it to McKean-Vlasov case at the end of this section.
\end{remark}

{
\begin{theorem}\label{existence-MFG-finite-fuel}
Under assumptions $\mathcal{A}_1$-$\mathcal{A}_6$ and the finite-fuel constraint $Z \in \widetilde{\mathcal{A}}^m(\mathbb{R})$, there exists a relaxed solution to (\ref{model-MFGs-singular-control}).
\end{theorem}}
\begin{proof}

From inequality (\ref{result-from-Markov-inequality}) in the proof of Proposition \ref{relative-compactness-general-result}, we see that for each $\mu\in\mathcal{P}_p(\widetilde{\mathcal{D}}(\mathbb{R}))$ and $\mathbb{P}\in\mathcal{R}(\mu)$, there exists a nonnegative function $k(\cdot)$ that is independent of $\mu$, such that $\mathbb{P}(\widetilde{w}_s({X},\delta)>\eta)\leq \frac{k(\delta)}{\eta}$ and $\lim_{\delta\rightarrow 0}k(\delta)=0$, where $\widetilde{w}_s$ is the modified oscillation function defined in (\ref{modified-oscillation}).

 Let us now define a set-valued map $\psi$ by
 \begin{align}\label{set-valued-function}
    \psi:~~&\mathcal{P}_p(\widetilde{\mathcal{D}}(\mathbb{R})) \rightarrow 2^{\mathcal{P}_p(\widetilde{\mathcal{D}}(\mathbb{R})},\nonumber\\
    & \mu \mapsto \{\mathbb{P}\circ X^{-1}: \mathbb{P}\in\mathcal{R}^*(\mu)\},
 \end{align}
and let
\begin{equation*}
    S=\left\{\mathbb{P}\in\mathcal{P}_p(\widetilde{\mathcal{D}}(\mathbb{R})):\textrm{for~ each}~ \eta>0,~ \mathbb{P}(\widetilde{w}_s({X},\delta)>\eta)\leq \frac{k(\delta)}{\eta}~\textrm{and}~E^{\mathbb{P}}\sup_{0\leq t\leq T}|X_t|^{\bar{p}}\leq C\right\}
\end{equation*}
where $C < \infty$ denotes the upper bound in (\ref{boundedness-expectation-any-order}). It can be checked that $S$ is non-empty, relatively compact, convex, and that $\psi(\mu)\subseteq S\subseteq\bar{S}$, for each $\mu\in\widetilde{\mathcal{D}}(\mathbb{R})$. Hence, $\psi: \bar{S}\rightarrow 2^{\bar{S}}$. Moreover, by Corollary \ref{existence-singular-control}, $\psi$ is nonempty-valued and upper hemi-continuous. Therefore, \cite[Corollary 17.55]{AB-2006} is applicable by embedding $\mathcal{P}_p(\widetilde{\mathcal{D}}(\mathbb{R}))$ into $\mathcal{M}(\widetilde{\mathcal{D}}(\mathbb{R}))$, the space of all bounded signed measures on $\widetilde{\mathcal{D}}(\mathbb{R})$ endowed with weak convergence topology. 
\end{proof}
\subsection{Existence in the general case} \label{Approximation}
In this section we establish the existence of a solution to MFGs with singular controls for general singular controls $Z \in \widetilde{\mathcal{A}}(\mathbb{R})$. For each $m$ and $\mu$, define $$\Omega^m=\widetilde{\mathcal{D}}(\mathbb{R})\times\widetilde{\mathcal{U}}(\mathbb R)\times\widetilde{\mathcal{A}}^m(\mathbb{R})$$ and denote by $\mathcal{R}^m(\mu)$ the control rules corresponding to $\Omega^m$ and $\mu$, that is, $\mathcal{R}^m(\mu)$ is the subset of probability measures in $\mathcal{R}(\mu)$ that are supported on $\Omega^m$.
Denote by \textbf{MFG}$^m$ the mean field games corresponding to $\Omega^m$. The preceding analysis showed that there exists a solution $\mathbb{P}^{m*}$ to \textbf{MFG}$^m$, for each $m$. {In what follows, $$\mu^{m*} :=\mathbb{P}^{m*}\circ X^{-1}.$$}
The next lemma shows that the sequence $\{\mathbb{P}^{m*}\}_{m\geq 1}$ is relatively compact; the subsequent one shows that any accumulation point is a control rule.

\begin{lemma}\label{boundedness-expectation-pm-zt-pbar}
Suppose $\mathcal{A}_1$, $\mathcal{A}_3$, $\mathcal{A}_4$ and $\mathcal{A}_6$ hold. Then there exists a constant $K < \infty$ such that $$\sup_m E^{\mathbb{P}^{m*}}|Z_T|^{\bar{p}}\leq K<\infty.$$ As a consequence, the sequence $\{\mathbb{P}^{m*}\}_{m\geq 1}$ is relatively compact in $\mathcal{W}_{p,\widetilde{\mathcal{D}}(\mathbb{R})\times\widetilde{\mathcal{U}}(\mathbb R)\times\widetilde{\mathcal{A}}(\mathbb{R})}$.
\end{lemma}
\begin{proof}
We recall that $c(\cdot)$ is bounded away from $0$. Hence, there exists a constant $C < \infty$ such that, for all $m\in\mathbb{N}$,
    \begin{equation}\label{estimate-pm-zt-pbar}
        E^{\mathbb{P}^{m*}}|Z_T|^{\bar{p}}\leq C\left(1+E^{\mathbb{P}^{m*}}|X_T|^{\bar{p}}\right)
    \end{equation}
and
    \begin{equation}\label{estimate-pm-xt-p}
        E^{\mathbb{P}^{m*}}|X_t|^{{p}}\leq C\left(1+E^{\mathbb{P}^{m*}}|Z_T|^{{p}}\right),~t\in[0,T].
    \end{equation}
Moreover,
    \begin{equation*}
        \begin{split}
            J(\mu^{m*},\mathbb{P}^{m*})=&~E^{\mathbb{P}^{m*}}\left[\int_0^T\int_U f(t,X_t,\mu^{m*}_t,u)\,Q_t(du)dt+g(X_T,\mu^{m*}_T)+\int_0^T h(t)\,dZ_t\right]\\
            \geq &~ -C\left(1+\int_0^T\int_{\mathbb{R}^d}|x|^p\,\mu^{m*}_t(dx)dt+E^{\mathbb{P}^{m*}}\int_0^T|X_t|^p\,dt+E^{\mathbb{P}^{m*}}\int_0^T\int_U|u|^p\,Q_t(du)dt\right.\\
            &\left.-E^{\mathbb{P}^{m*}}|X_T|^{\bar{p}}+\int_{\mathbb{R}^d}|x|^p\,\mu^{m*}_T(dx)+E^{\mathbb{P}^{m*}}\left|\int_0^Th(t)\,dZ_t\right|\right)~~~(\textrm{by assumption}~\mathcal{A}_3)\\
            \geq &~-C\left(1+\int_0^T\int_{\mathbb{R}^d}|x|^p\,\mu^{m*}_t(dx)dt+E^{\mathbb{P}^{m*}}\int_0^T|X_t|^p\,dt+E^{\mathbb{P}^{m*}}\int_0^T\int_U|u|^p\,Q_t(du)dt\right.\\
            &~\left.+\int_{\mathbb{R}^d}|x|^p\,\mu^{m*}_T(dx)+E^{\mathbb{P}^{m*}}\left|\int_0^Th(t)\,dZ_t\right|-E^{\mathbb{P}^{m*}}|Z_T|^{\bar{p}}\right)~~~(\textrm{by } (\ref{estimate-pm-zt-pbar})).
        \end{split}
    \end{equation*}
Now choose any $\mathbb{P}_0\in\mathcal{R}^m(\mu^{m*})$ such that $\sup_mJ(\mu^{m*},\mathbb{P}_0) < \infty$ (e.g. $\mathbb{P}_0\in\mathcal{R}(\mu^{m,*})$ such that $\mathbb{P}_0(Q|_{[0,T]}\equiv \delta_{\widetilde{u}}(du)dt|_{[0,T]},Z\equiv 0)=1$ for some $\widetilde{u}\in U$). Then,
    \begin{equation}\label{estimate-pm-zt-pbar-(1)}
        \begin{split}
            E^{\mathbb{P}^{m*}}|Z_T|^{\bar{p}}\leq&~ J(\mu^{m*},\mathbb{P}^{m*})+C\left(1+E^{\mathbb{P}^{m*}}\left|\int_0^Th(t)\,dZ_t\right|+E^{\mathbb{P}^{m*}}\int_0^T|X_t|^p\,dt+E^{\mathbb{P}^{m*}}|X_T|^p\right)\\
            \leq&~J(\mu^{m*},\mathbb{P}_0)+C\left(1+E^{\mathbb{P}^{m*}}|Z_T|+E^{\mathbb{P}^{m*}}|Z_T|^p\right)~~~~(\textrm{by }(\ref{estimate-pm-xt-p})\textrm{ and the optimality of }\mathbb{P}^{m*})\\
            \leq &~C\left(1+E^{\mathbb{P}^{m*}}|Z_T|+E^{\mathbb{P}^{m*}}|Z_T|^p\right).
        \end{split}
    \end{equation}
Since the measure $\mathbb{P}^{m*}$ is supported on $\Omega^m$, we see that $E^{\mathbb{P}^{m*}}|Z_T|^{\bar{p}}$ is finite, for each $m$. In order to see that there exists a uniform upper bound on $E^{\mathbb{P}^{m*}}|Z_T|^{\bar{p}}$, notice that, independently of $m$ we can choose $M > 0$ large enough such that
\[	
	E^{\mathbb{P}^{m*}}|Z_T|^{{p_0}} \leq M + \frac{1}{4C} E^{\mathbb{P}^{m*}}|Z_T|^{\bar{p}}\quad (p_0=1,p)
\]
Together with (\ref{estimate-pm-zt-pbar-(1)}) this yields,
    \begin{equation*}
        E^{\mathbb{P}^{m*}}|Z_T|^{\bar{p}}\leq 2C(1+M):=K.
    \end{equation*}
By \cite[Definition 6.8]{Villani-2009} and Proposition \ref{relative-compactness-general-result}, the relative compactness of $\{\mathbb{P}^{m*}\}_{m\geq 1}$ follows.
\end{proof}

The previous lemma shows that the sequence $\{\mathbb{P}^{m*}\}_{m\geq 1}$ has an accumulation point $\mathbb{P}^*$. {Let $\mu^*=\mathbb{P}^*\circ X^{-1}$. Clearly, $\mu^{m*}\rightarrow \mu^*$ in $\mathcal{W}_p$ along a subsequence.} The following result is an immediate corollary to Proposition \ref{admissibility-limit-P-proposition} (see Remark \ref{remark-on-admissibility}).

\begin{lemma}
Suppose that $\mathcal{A}_1$ and $\mathcal{A}_3$-$\mathcal{A}_6$ hold, 
let $\mathbb{P}^*$ be an accumulation point of the sequence $\{\mathbb{P}^{m*}\}_{m\geq 1}$. Then, $\mathbb{P}^*\in\mathcal{R}(\mu^*)$.
\end{lemma}
{The next theorem establish the existence of relaxed MFGs solution to (\ref{model-MFGs-singular-control}) in the general case, i.e. it proves Theorem \ref{existence}.}
\begin{theorem}
Suppose $\mathcal{A}_1$-$\mathcal{A}_6$ hold. Then
$\mathbb{P}^*\in\mathcal{R}^*(\mu^*)$, i.e., for each $\mathbb{P}\in\mathcal{R}(\mu^*)$ it holds that $$J(\mu^*,\mathbb{P}^*)\leq J(\mu^*,\mathbb{P}).$$
\end{theorem}
\begin{proof}
It is sufficient to prove that $J(\mu^*,\mathbb{P}^*)\leq J(\mu^*,\mathbb{P})$ for each $\mathbb{P}\in\mathcal{R}(\mu^*)$ with $J(\mu^*,\mathbb{P})<\infty$.

By Proposition \ref{result-from-Karoui}, there exists a filtered probability space $(\bar{\Omega},\bar{\mathcal{F}},\bar{\mathcal{F}}_t,\bar{\mathbb{P}})$ on which random variables $(\bar{X},\bar{Q},\bar{Z},M)$ are defined such that $\mathbb{P}=\bar{\mathbb{P}}\circ (\bar{X},\bar{Q},\bar{Z})^{-1}$ and
    \begin{equation}\label{X-bar-SDE}
        d\bar{X}_t=\int_U b(t,\bar{X}_t,{\mu}^*_t,u)\,\bar{Q}_t(du)dt+\int_U \sigma(t,\bar{X}_t,{\mu}^*_t,u)M(du,dt)+c(t)\,d\bar{Z}_t,
    \end{equation}
where $M$ is a martingale measure with intensity $\bar{Q}$.
Using the same argument as in the proof of Lemma \ref{boundedness-expectation-pm-zt-pbar} we see that,
    \begin{equation}\label{boundedness-expectation-Pbar-zt-pbar}
        E^{{\mathbb{P}}}{Z}_T^{\bar{p}}=E^{\bar{\mathbb{P}}}\bar{Z}_T^{\bar{p}}<\infty.
    \end{equation}
Define $\mathbb{P}^m=\bar{\mathbb{P}}\circ (\bar{X}^m,\bar{Q},\bar{Z}^m)\in\mathcal{R}^m(\mu^{m*})$, such that $\bar{X}^m$ is the unique strong solution to
    \begin{equation}\label{X-bar-m-SDE}
        d\bar{X}^m_t=\int_U b(t,\bar{X}^m_t,{\mu}^{m*}_t,u)\,\bar{Q}_t(du)dt+\int_U \sigma(t,\bar{X}^m_t,{\mu}^{m*}_t,u)M(du,dt)+c(t)\,d\bar{Z}^m_t,
    \end{equation}
where for each $\bar{\omega}\in\bar{\Omega}$,
\begin{equation*}
        \bar{Z}^m_t(\bar{\omega})=\left\{
        \begin{array}{ll}
        \bar{Z}_t(\bar{\omega}),&\textrm{ if }t<\tau^m(\bar{\omega})\\
        m,&\textrm{ if }t\geq \tau^m(\bar{\omega}),\\
        \end{array}
        \right.
\end{equation*}
with $\tau^m(\bar{\omega})=\inf\{t:\bar{Z}_t{(\bar{\omega})}>m\}$. Similarly, we can define $Z^m$.
Furthermore, if ${Z}$ is $\widetilde{\mathcal{A}}^m(\mathbb{R})$ valued, we have ${Z}={Z}^m$. Hence,
    \begin{equation}
        \begin{split}
            &~E^{\bar{\mathbb{P}}}\sup_{0\leq t\leq T}\left|\int_0^t c(s)\,d\bar{Z}_s-\int_0^tc(s)\,d\bar{Z}^m_s\right|\\
            =&~E^{{\mathbb{P}}}\sup_{0\leq t\leq T}\left|\int_0^t c(s)\,d{Z}_s-\int_0^tc(s)\,d{Z}^m_s\right|\\
            =&~\int_{\widetilde{\mathcal{A}}(\mathbb{R})\backslash\widetilde{\mathcal{A}}^m(\mathbb{R})}\sup_{0\leq t\leq T}\left|\int_0^tc(s)\,d{Z}_s({\omega})-\int_0^tc(s)\,d{Z}^m_s({\omega})\right|{\mathbb{P}}(d{\omega}).
        \end{split}
    \end{equation}
 By H\"older's inequality,
    \begin{equation*}
    \begin{split}
        &~\int_{\widetilde{\mathcal{A}}(\mathbb{R})\backslash\widetilde{\mathcal{A}}^m(\mathbb{R})}\sup_{0\leq t\leq T}\left|\int_0^tc(s)\,d{Z}_s(\omega)-\int_0^tc(s)\,d{Z}^m_s({\omega})\right|{\mathbb{P}}(d{\omega})\\
        \leq&~\left|\int_{\widetilde{\mathcal{A}}(\mathbb{R})\backslash\widetilde{\mathcal{A}}^m(\mathbb{R})}\int_0^Tc(t)\,d{Z}_t({\omega}){\mathbb{P}}(d{\omega})+\int_{\widetilde{\mathcal{A}}(\mathbb{R})\backslash\widetilde{\mathcal{A}}^m(\mathbb{R})}\int_0^Tc(t)\,d{Z}^m_t({\omega}){\mathbb{P}}(d{\omega})\right|\\
        \leq &~ C \left(E^{{\mathbb{P}}}{Z}_T^{{p}}\right)^{\frac{1}{{p}}}{\mathbb{P}}(\widetilde{\mathcal{A}}(\mathbb{R})\backslash\widetilde{\mathcal{A}}^m(\mathbb{R}))^{1-\frac{1}{{p}}}+C \left(E^{{\mathbb{P}}}({Z}^m_T)^{{p}}\right)^{\frac{1}{{p}}}{\mathbb{P}}(\widetilde{\mathcal{A}}(\mathbb{R})\backslash\widetilde{\mathcal{A}}^m(\mathbb{R}))^{1-\frac{1}{{p}}}\\
        \leq &~ C \left(E^{{\mathbb{P}}}{Z}_T^{{p}}\right)^{\frac{1}{{p}}}{\mathbb{P}}(\widetilde{\mathcal{A}}(\mathbb{R})\backslash\widetilde{\mathcal{A}}^m(\mathbb{R}))^{1-\frac{1}{{p}}}.
    \end{split}
    \end{equation*}
Since $\widetilde{\mathcal{A}}^m(\mathbb{R})\uparrow \widetilde{\mathcal{A}}(\mathbb{R})$ implies ${\mathbb{P}}(\widetilde{\mathcal{A}}(\mathbb{R})\backslash\widetilde{\mathcal{A}}^m(\mathbb{R}))\rightarrow 0$ we get,
    \begin{equation}\label{approximation-to-0-integral-c-zm-c-z}
        E^{\bar{\mathbb{P}}}\sup_{0\leq t\leq T}\left|\int_0^t c(s)\,d\bar{Z}_s-\int_0^tc(s)\,d\bar{Z}^m_s\right|\rightarrow 0.
    \end{equation}
Similarly,
    \begin{equation}\label{approximation-to-0-integral-h-zm-h-z}
        E^{\bar{\mathbb{P}}}\left|\int_0^T h(t)\,d\bar{Z}_t-\int_0^Th(t)\,d\bar{Z}^m_t\right|\rightarrow 0.
    \end{equation}
By (\ref{X-bar-SDE}), (\ref{X-bar-m-SDE}) and \eqref{approximation-to-0-integral-c-zm-c-z}, the Lipschitz continuity of $b$ and $\sigma$ in $x$ and $\mu$ and the Burkholder-Davis-Gundy inequality, standard estimate of SDE yields that
    \begin{equation}\label{convergence-X-m-to-X-(4)}
        \lim_{m\rightarrow\infty}E^{\bar{\mathbb{P}}}\sup_{0\leq t\leq T}\left|\bar{X}^m_t-\bar{X}_t\right|= 0.
    \end{equation}
By (\ref{approximation-to-0-integral-h-zm-h-z}), (\ref{convergence-X-m-to-X-(4)}), $\mu^{m*}\rightarrow\mu^*$ in $\mathcal{W}_{p,(\widetilde{\mathcal{D}}(\mathbb{R}),d_{M_1})}$ and the same arguments as in the proof of Lemma \ref{continuity-J-mu-P}, we get
    \begin{equation*}
    \begin{split}
       &~ E^{\bar{\mathbb{P}}} \left(\int_0^T f(t,\bar{X}^m_t,\mu^{m*}_t,u)\,\bar{Q}_t(du)dt+g(\bar{X}^m_T,\mu^{m*}_T)+\int_0^Th(t)\,d\bar{Z}^m_t\right)\\
       \rightarrow&~E^{\bar{\mathbb{P}}} \left(\int_0^T f(t,\bar{X}_t,\mu^*_t,u)\,\bar{Q}_t(du)dt+g(\bar{X}_T,\mu^*_T)+\int_0^Th(t)\,d\bar{Z}_t\right).
    \end{split}
    \end{equation*}
This shows that
\[
	J(\mu^{m*},\mathbb{P}^m)\rightarrow J(\mu^*,\mathbb{P}).
\]
%
Moreover, by Remark \ref{joint-LSC-J}, $\liminf_{m\rightarrow\infty}J(\mu^{m*},\mathbb{P}^{m*})\geq J(\mu^*,\mathbb{P}^*)$.
Hence,
    \[
        J(\mu^*,\mathbb{P})= \lim_{m\rightarrow\infty}J(\mu^{m*},\mathbb{P}^m)\geq \liminf_{m\rightarrow\infty}J(\mu^{m*},\mathbb{P}^{m*})\geq J(\mu^*,\mathbb{P}^*).
    \]
\end{proof}
\subsection{Related McKean-Vlasov stochastic singular control problem}\label{McKean-Vlasov-control}
MFGs and control problems of McKean-Vlasov type are compared in \cite{carmona-delarue-lachapelle}. The literatures on McKean-Vlasov singular control focus on necessary conditions for optimality; the existence of optimal control is typically assumed; see e.g. \cite{hu-oksendal-sulem-2014}. With the above method for MFGs, we can also establish the existence of an optimal control to the following McKean-Vlasov stochastic singular control problem:
\begin{equation}\label{cost-MV}
    \min_{u,Z} J(u,Z)=\min_{u,Z} E\left[\int_0^T f(t,X_t,Law(X_t),u_t)\,dt+g(X_T,Law(X_T))+\int_0^T h(t)\,dZ_t\right]
\end{equation}
subject to
\begin{equation}\label{state-MV}
\begin{split}
    dX_t=b(t,X_t,Law(X_t),u_t)\,dt+\sigma(t,X_t,Law(X_t),u_t)\,dW_t+c(t)\,dZ_t, ~t\in[0,T].
\end{split}
\end{equation}
To this end, we need to introduce relaxed controls and control rules similar to Section \ref{definition-assumption}.

%
\begin{definition}\label{def-relaxed-control-MV}
We call $(\Omega,\mathcal{F},\{\mathcal{F}_t,t\in\mathbb{R}\},\mathbb{P},X,\underline{Q},Z)$ a relaxed control to McKean-Vlasov stochastic singular control problem (\ref{cost-MV})-(\ref{state-MV}) if it satisfies items 1, 2 and 3 in Definition \ref{relaxed-control} and
\begin{itemize}
\item[4'] $\left(\mathcal{M}^{\mathbb{P},\phi},\{\mathcal{F}_t,t\geq 0\},\mathbb{P}\right)$ is a well defined continuous martingale, where
    \begin{equation}\label{martingale-MV}
    \begin{split}
        \mathcal{M}^{\mathbb{P},\phi}_t=&~\phi(X_t)-\int_0^t\int_U\phi'(X_s)b(s,X_s,\mathbb{P}\circ X_s^{-1},u)\,\underline{Q}_s(du)ds \\
        &-\frac{1}{2}\int_0^t\int_U\phi'(X_s)a(s,X_s,\mathbb{P}\circ X_s^{-1},u)\,\underline{Q}_s(du)ds\\
        &-\int_0^t\phi'(X_{s-})c(s)\,dZ_s-\sum_{0\leq s\leq t}\left(\phi(X_{s})-\phi(X_{s-})-\phi'(X_{s-})\Delta X_s\right),~t\in[0,T].
    \end{split}
    \end{equation}
   \end{itemize}
\end{definition}
For each relaxed control $r=(\Omega,\mathcal{F},\{\mathcal{F}_t,t\in\mathbb R\},\mathbb{P},X,\underline{Q},Z)$, we define the corresponding cost functional by
\begin{equation}\label{cost-relaxed-control-control-rule}
    J(r)=E^{\mathbb{P}}\left[\int_0^T\int_U f\left(t,X_t,\mathbb{P}\circ X^{-1}_t,u\right)\,\underline{Q}_t(du)dt+g\left(X_T,\mathbb{P}\circ X^{-1}_T\right)+\int_0^T h(t)\,dZ_t\right].
\end{equation}
We still denote by $\Omega:=\widetilde{\mathcal D}(\mathbb R)\times\widetilde{\mathcal U}(\mathbb R)\times\widetilde{\mathcal A}(\mathbb R)$ the canonical space, $\mathcal F_t$ the canonical filtration and $(X,Q,Z)$ the coordinate projections with the associated predictable disintegration $Q^o$, as introduced in Section \ref{definition-assumption}. The notion of control rules can be defined similarly as that in Definition \ref{control-rule}. Denote by $\mathcal{R}$ all the control rules. For $\mathbb{P}\in\mathcal{R}$, the corresponding cost functional is defined as in (\ref{cost-relaxed-control-control-rule}).

Using straightforward modifications of arguments given in the proof of \cite[Proposition 2.6]{Haussmann-Suo-1995} we see that our
%
 optimization problems over relaxed controls and over control rules are equivalent. Once the optimal control rule is established, under the same additional assumption as in Remark \ref{existence-strict-MFG}, we can establish a strict optimal control from the optimal control rule. 
The next two theorems prove the existence of an optimal control under a finite-fuel constraint. The existence results can then be extended to the general unconstraint case. We do not give a formal proof as the arguments are exactly the same as in the preceding subsection.
\begin{theorem}
Suppose $\mathcal{A}_4$, $\mathcal{A}_5$ hold and $\mathcal{A}_1$ holds without Lipschitz continuity of $b$ and $\sigma$ on $x$. Under a finite-fuel constraint on the singular controls,
$\mathcal{R}\neq\emptyset$.
\end{theorem}
\begin{proof}
For each $\mu\in\mathcal{P}_p(\widetilde{\mathcal{D}}(\mathbb{R}))$, there exists a solution to the martingale problem $\mathcal{M}^{\mu,\phi}$, where $\mathcal{M}^{\mu,\phi}$ is defined in \eqref{martingale-problem-MFGs}.
Thus, we define a set-valued map $\Phi$ on $\mathcal{P}_p(\widetilde{\mathcal{D}}(\mathbb{R}))$ with non-empty convex images by
    \[
        \Phi: \mu\rightarrow \{\mathbb{P}\circ X^{-1}: \mathbb P\in\mathcal R(\mu)\},
    \]
where $\mathcal R(\mu)$ is the control rule with $\mu$ as in the previous section.

The compactness of $\Phi(\mu)$ for each $\mu\in\mathcal P_p(\widetilde{\mathcal D}(\mathbb R))$ and the upper hemi-continuity of $\Phi$ are results of the compactness of $\mathcal R(\mu)$ for each $\mu\in\mathcal P_p(\widetilde{\mathcal D}(\mathbb R))$ and upper hemi-continuity of $\mathcal R(\cdot)$, respectively, which are direct results of Corollary \ref{continuity-R}.\footnote{{Note that we only need upper hemi-continuity of $\mathcal R(\cdot)$, so Lipschtiz assumptions on $b$ and $\sigma$ are not necessary.}} By analogy to the proof of Theorem \ref{existence-MFG-finite-fuel} we can define a non-empty, compact, convex set $\bar S \subset \mathcal{P}_p(\widetilde{\mathcal{D}}(\mathbb{R}))$ such that $\Phi: \bar{S} \to 2^{\bar S}$. Hence, $\Phi$ has a fixed point, due to \cite[Corollary 17.55]{AB-2006}.
%
\end{proof}
\begin{theorem}\label{existence-MV-control}
Suppose $\mathcal{A}_3$-$\mathcal{A}_6$ hold {and that $\mathcal{A}_1$ holds without Lipschitz assumptions on $b$ and $\sigma$ in $x$, and that $\mathcal{A}_2$ holds with the continuity of $f$ and $g$ being replaced by lower semi-continuity.}
Under a finite-fuel constraint, there exist an optimal control rule, that is, there exists $\mathbb{P}^*\in\mathcal{R}$ such that $$J(\mathbb{P}^*)\leq J(\mathbb{P}) \quad \mbox{for all} \quad \mathbb{P}\in\mathcal{R}.$$
\end{theorem}
\begin{proof}
It is sufficient to prove $\mathcal{R}$ is compact and $J$ is lower semi-continuous. The former one can be achieved by the same way to Corollary \ref{continuity-R}. As for the lower semi-continuity, note that $f$ and $g$ can be approximated by continuous functions $f_N$ and $g_N$ increasingly. For $f_N$ and $g_N$, by the same way as that in the proof of Lemma \ref{continuity-J-mu-P}, one has
\begin{equation*}
    \begin{split}
    &~\liminf_{n\rightarrow\infty}E^{\mathbb P^n}\left[\int_0^T\int_U f_N(t,X_t,\mathbb P^n\circ X^{-1}_t,u)\,Q_t(du)dt+g_N(X_T,\mathbb P^n\circ X^{-1}_T)+\int_0^T h(t)\,dZ_t\right]\\
    \rightarrow &~E^{\mathbb P}\left[\int_0^T\int_U f_N(t,X_t,\mathbb P\circ X^{-1}_t,u)\,Q_t(du)dt+g_N(X_T,\mathbb P\circ X^{-1}_T)+\int_0^T h(t)\,dZ_t\right].
    \end{split}
\end{equation*}
Thus, monotone convergence implies the lower semi-continuity of $J$.
\end{proof}

\section{MFGs with regular controls and MFGs with singular controls}

In this section we establish two approximation results for a class of MFGs with singular controls under finite-fuel constraints. For the reasons outlined in Remark \ref{remark-on-discontinuity-T} below we restrict ourselves to MFGs without terminal cost or singular control cost. More precisely, we consider MFGs with singular controls of the form:
  \begin{equation} \label{MFG1}
          \left\{
        \begin{array}{ll}
        1.&\textrm{ fix a deterministic measure }\mu\in\mathcal{P}_p(\widetilde{\mathcal{D}}_{0,T+\epsilon}(\mathbb{R}));\\
        2.&\textrm{ solve the corresponding stochastic singular control problem}:\\
          &\inf_{u,Z } E\left[\int_0^Tf(t,X_t,{\mu}_t,u_t)dt\right] \\
        &\textrm{subject to} \\
        &dX_t=b(t,X_t,{\mu}_t,u_t)\,dt+\sigma(t,X_t,{\mu}_t,u_t)dW_t+c(t)\,dZ_t,~t\in[0,T+\epsilon];\\
        3.&\textrm{solve the fixed point problem: }{\mu={Law}(X),}
        \end{array} \right.
    \end{equation}
   for some fixed $\epsilon>0$ under the finite-fuel constraint $Z\in \widetilde{{\mathcal{A}}}^m_{0,T}(\mathbb{R})$. The reason we define the state process on the time interval $[0,T+\epsilon]$ is that we approximate the singular controls by absolutely continuous ones that are most naturally regarded as elements of $\widetilde{\cal D}_{0,T+\epsilon}({\mathbb R})$ rather than $\widetilde{\cal D}_{0,T}({\mathbb R})$.

%

{\subsection{Solving MFGs with singular controls using MFGs with regular controls}\label{Sec4.1}}\label{regular-approximate-singular}
In this section we establish an approximation of (relaxed) solutions results for the MFGs (\ref{MFG1}) under a finite-fuel constraint by (relaxed) solutions to MFGs with only regular controls.
To this end, we associate with each singular control $Z\in\widetilde{{\mathcal{A}}}^m_{0,T}(\mathbb{R})$ the sequence of absolutely continuous controls
    \begin{equation}\label{approximation-of-singular-control}
        Z^{[n]}_t=n\int_{(t-\frac{1}{n})}^t Z_s\,ds \quad \left(t\in\mathbb{R}, n \in \mathbb{N} \right).
    \end{equation}
  Then, $Z^{[n]}\in\widetilde{\cal A}^m_{0,T+\epsilon}({\mathbb R})$ for all sufficiently large $n \in \mathbb N$. Since each $Z^{[n]}$ is absolutely continuous and $Z$ is c\`adl\`ag we cannot expect convergence of $Z^n$ to $Z$ in the Skorokhod $J_1$ topology in general. However, by Proposition \ref{characterization-convergence-M1} (3.) and the discussion before Proposition \ref{modified-M1} we do know that
 \[
 	Z^{[n]} \to Z \quad \mbox{a.s. in} \quad \left(\widetilde{\cal D}_{0,T+\epsilon}({\mathbb R}), d_{M_1} \right).
\]

   For each $n$, we consider the following finite-fuel constrained MFGs denoted by \textbf{MFG$^{[n]}$}:
    \begin{equation}\label{cost-fun-n}
          \left\{
        \begin{array}{ll}
        1.&\textrm{ fix a deterministic measure }\mu\in\mathcal{P}_p(\widetilde{\mathcal{D}}_{0,T+\epsilon}(\mathbb R));\\
        2.&\textrm{ solve the corresponding stochastic control problem}:\\
          &\inf_{u,Z}E\left[\int_0^Tf(t,X^{[n]}_t,{\mu}_t,u_t)dt\right] \\
          &\textrm{subject to}\\
	      &dX^{[n]}_t=b(t,X^{[n]}_t,{\mu}_t,u_t)dt+\sigma(t,X^{[n]}_t,{\mu}_t,u_t)dW_t+c(t)\,dZ^{[n]}_t,~ t\in[0,T+\epsilon] \\
	      &X^{[n]}_{0} = 0 \\
	      &Z^{[n]}_t=n\int_{(t-\frac{1}{n})}^t Z_s\,ds;\\
        3. &\textrm{solve the fixed point problem: }{\mu={Law}(X^{[n]})}.
    \end{array} \right.
	\end{equation}
\begin{definition}\label{relaxed-control-zn}
We call the vector $r^n=(\Omega,\mathcal{F},\{\mathcal{F}_t,t\in\mathbb R\},\mathbb{P},X,\underline Q,Z^{[n]})$ a {\sl relaxed control} with respect to $\mu$ for some $\mu\in\mathcal{P}_p(\widetilde{\mathcal{D}}_{0,T+\epsilon}(\mathbb{R}))$
if $(\Omega,\mathcal{F},\{\mathcal{F}_t,t\in\mathbb R\},\mathbb{P},X,\underline Q,Z)$ satisfies 1.-3. in Definition \ref{relaxed-control} with item $4$ being replaced by
    \begin{itemize}
        \item[$4'$.] $X$ is a $\{\mathcal{F}_t,t\in\mathbb R\}$ adapted stochastic process and $X\in\widetilde{\mathcal{D}}_{0,T+\epsilon}(\mathbb{R})$ such that for each $\phi\in \mathcal{C}^2_b(\mathbb{R}^d)$, $\mathcal{M}^{[n],\mu,\phi}$ is a well defined $\mathbb{P}$ continuous martingale, where
            \begin{equation}\label{martingale-problem-n}
            \begin{split}
                \mathcal{M}^{[n],\mu,\phi}_t:=&~\phi(X_t)-\int_0^t\int_U\mathcal{L}\phi(s,X_s,\mu_s,u)\,\underline Q_s(du)ds-\int_0^t(\partial_x\phi(X_s))^{\top}c(s)\,dZ^{[n]}_s,
             \end{split}
             \end{equation}
             with $\mathcal{L}$ defined as in Definition \ref{control-rule-mu}.
    \end{itemize}

The probability measure $\mathbb{P}$ is called a control rule if $(\Omega,\mathcal{F},\{\mathcal{F}_t,t\in\mathbb R\},\mathbb{P},X,{Q}^o,Z^{[n]})$ is a relaxed control with $(\Omega,\mathcal{F},\{\mathcal{F}_t,t\in\mathbb R\})$ being the filtered canonical space with $$\Omega:={\widetilde{\mathcal{D}}}_{0,T+\epsilon}(\mathbb{{R}})\times \widetilde{{\cal U}}_{0,T+\epsilon}(\mathbb R) \times \widetilde{\mathcal{A}}^m_{0,T}(\mathbb{R})$$ and $(X,Q,Z)$ being the coordinate projections on $(\Omega,\mathcal{F},\{\mathcal{F}_t,t\in\mathbb R\})$ and ${Q}^o$ being the disintegration of $Q$ as in Section \ref{disintegration}.
\end{definition}

\begin{remark}\label{remark-on-discontinuity-T}
If $Z$ is discontinuous at $T$, then $Z^{[n]}$ may not converge to $Z$ in $\widetilde{{\cal D}}_{0,T}(\mathbb{R})$ but only in $\widetilde{{\cal D}}_{0,T+\epsilon}(\mathbb{R})$. Likewise, the associated sequence of the state processes may only converge in $\widetilde{{\cal D}}_{0,T+\epsilon}(\mathbb{R})$.
The possible discontinuity at the terminal time $T$ is also the reason why there is no terminal cost and no cost from singular control in this section. If we assume that $T$ is always a continuous point, then terminal costs and costs from singular controls are permitted. In this case, one may as well allow unbounded singular controls.
\end{remark}

For each fixed $n$ and $\mu$, denote by $\mathcal{R}^{[n]}(\mu)$ the set of all the control rules for \textbf{MFG$^{[n]}$}, and define the cost functional corresponding to the control rule $\mathbb{P}\in\mathcal{R}^{[n]}(\mu)$ by
    \[
        J^{[n]}(\mu,\mathbb{P})=E^{\mathbb{P}}\left(\int_0^T\int_U f(t,X_t,\mu_t,u)\,Q_t(du)dt\right).
    \]
For each fixed $n$ and $\mu$, denote by $\mathcal{R}^{[n]*}(\mu)$ the set of all the optimal control rules. We can still check that $$\inf\limits_{\textrm{relaxed control }r^n}J^{[n]}(\mu,r^n)=\inf\limits_{\mathbb{P}\in\mathcal{R}^{[n]}(\mu)}J^{[n]}(\mu,\mathbb{P}),$$ which implies we can still restrict ourselves to control rules in analyzing \textbf{MFG$^{[n]}$}.

The proof of the following theorem is very similar to that of Theorem \ref{existence-MFG-finite-fuel} and is hence omitted.

\begin{theorem}
Suppose $\mathcal{A}_1$-$\mathcal{A}_6$ hold.
For each $n$, there exists a relaxed solution $\mathbb{P}^{[n]}$ to \textbf{MFG$^{[n]}$}.
\end{theorem}
By Proposition \ref{relative-compactness-general-result}, the sequence $\left\{\mathbb{P}^{[n]}\right\}_{n\geq 1}$ is relatively compact. Denote its limit (up to a subsequence) by $\mathbb{P}^*$ and set $\mu^*=\mathbb{P}^*\circ X^{-1}$. Then, $\mu^*$ is the limit of $\mu^{[n]}:=\mathbb{P}^{[n]}\circ X^{-1}$. The following lemma shows that $\mathbb{P}^*$ is admissible.
\begin{lemma}\label{admissibility-limit-P}
    Suppose $\mathcal{A}_1$-$\mathcal{A}_2$, $\mathcal{A}_4$-$\mathcal{A}_6$ hold. Then $\mathbb{P}^*\in\mathcal{R}(\mu^*)$.
\end{lemma}
\begin{proof}
By Proposition \ref{transfer-X-to-Y} there exists, for each $n$, a $\{\mathcal{F}_t,0\leq t\leq T+\epsilon\}$ adapted continuous process $Y^n$, such that
    \[
        \mathbb{P}^{[n]}\left({X}_{t}=Y^n_{t}+\int_0^{t}c(s)\,d{Z}^{[n]}_{s},~t\in[0,T+\epsilon]\right)=1.
    \]
Arguing as in the proof of Proposition \ref{admissibility-limit-P-proposition}, there exists a probability space $(\widetilde{\Omega},\widetilde{\mathcal{F}},\mathbb{Q})$ supporting random varibales $(\widetilde{X}^n,\widetilde{Y}^n,\widetilde{Q}^n,\widetilde{Z}^n)$ and $(\widetilde{X},\widetilde{Y},\widetilde{Q},\widetilde{Z})$ such that
	$(\widetilde{X}^n,\widetilde{Y}^n,\widetilde{Q}^n,\widetilde{Z}^n) \to (\widetilde{X},\widetilde{Y},\widetilde{Q},\widetilde{Z}) ~ \mathbb{Q}\mbox{-a.s.}$
and
\[
    \mathbb{P}^{[n]}\circ (X,Y^n,Q,Z)^{-1}=\mathbb{Q}\circ (\widetilde{X}^n,\widetilde{Y}^n,\widetilde{Q}^n,\widetilde{Z}^n)^{-1},
\]
which implies
    \begin{equation}\label{widetilde-x-n-approximation}
        \mathbb{Q}\left(\widetilde{X}^n_{t}=\widetilde{Y}^n_{t}+\int_0^{t}c(s)\,d\widetilde{Z}^{[n],n}_s,~t\in[0,T+\epsilon]\right)=1,
    \end{equation}
where $\widetilde{Z}^{[n],n}_t=n\int_{(t-1/n)}^t \widetilde{Z}^n_s\,ds$. For each fixed $\widetilde{\omega}\in\widetilde{\Omega}$ and for each $t$ which is a continuous point of $\widetilde{Z}(\widetilde{\omega})$,  by (\ref{local-uniform-convergence-at-continuity}) in Proposition \ref{characterization-convergence-M1}, we have
\begin{equation*}
\begin{split}
    \left|n\int_{t-\frac{1}{n}}^t\widetilde{Z}^n_s(\widetilde{\omega})\,ds-\widetilde{Z}_t(\widetilde{\omega})\right|\leq &~ n\int_{t-\frac{1}{n}}^t|\widetilde{Z}^n_s(\widetilde{\omega})-\widetilde{Z}_s(\widetilde{\omega})|\,ds+n\int_{t-\frac{1}{n}}^t|\widetilde{Z}_s(\widetilde{\omega})-\widetilde{Z}_t(\widetilde{\omega})|\,ds\\
\leq&~\sup_{t-\frac{1}{n}\leq s\leq t}|\widetilde{Z}^n_s(\widetilde{\omega})-\widetilde{Z}_s(\widetilde{\omega})|+\sup_{t-\frac{1}{n}\leq s\leq t}|\widetilde{Z}_s(\widetilde{\omega})-\widetilde{Z}_t(\widetilde{\omega})|\\
\rightarrow&~ 0.
\end{split}
\end{equation*}
Then (\ref{widetilde-x-n-approximation}) and right-continuity of the path yield that
\begin{equation}\label{widetilde-x-to-be-approximated}
            \mathbb{Q}\left(\widetilde{X}_{t}=\widetilde{Y}_{t}+\int_0^{t}c(s)\,d\widetilde{Z}^{}_s,~t\in[0,T+\epsilon]\right)=1.
\end{equation}
The desired result can be obtained by the same proof as Proposition \ref{admissibility-limit-P-proposition}.

\end{proof}
\begin{remark}
In the above proof, the local uniform convergence near a continuous point is necessary. As stated in Proposition \ref{characterization-convergence-M1}, this is a direct consequence of the convergence in the $M_1$ topology. Local uniform convergence cannot be guaranteed in the Meyer-Zheng topology. For Meyer-Zheng topology, we only know that convergence is equivalent to convergence in Lebesgue measure but we do not have uniform convergence in general.
\end{remark}

We are now ready to state and prove the main result of this section.
\begin{theorem}\label{limit-is-solution-MFG}
    Suppose $\mathcal{A}_1$-$\mathcal{A}_6$ hold. Then $\mathbb{P}^*$ is a relaxed solution to the MFG \eqref{MFG1}.
\end{theorem}
\begin{proof}
For each $\mathbb{P}\in\mathcal{R}(\mu^*)$ such that $J(\mu^*,\mathbb{P})<\infty$, on an extension $(\widetilde{\Omega},\widetilde{\mathcal{F}},\{\widetilde{\mathcal{F}}_t,t\in\mathbb R\},\widetilde{\mathbb{P}})$ we have,
    \begin{equation*}
      d\widetilde{X}_t=\int_Ub(t,\widetilde{X}_t,\mu^*_t,u)\,\widetilde{Q}_t(du)dt+\int_U\sigma(t,\widetilde{X}_t,\mu^*_t,u)\,\widetilde{M}(du,dt)+c(t)\,d\widetilde{Z}_t,
    \end{equation*}
and $\mathbb{P}=\widetilde{\mathbb{P}}\circ (\widetilde{X},\widetilde{Q},\widetilde{Z})^{-1}$. Let $\widetilde{Z}^{[n]}_t=n\int_{t-1/n}^t\widetilde{Z}_s\,ds$. By the Lipschitz continuity of the coefficient $b$ and $\sigma$, there exists a unique strong solution $X^n$ to the following SDE on $(\widetilde{\Omega},\widetilde{\mathcal{F}},\{\widetilde{\mathcal{F}}_t,t\in\mathbb R\},\widetilde{\mathbb{P}})$:
    \begin{equation*}
      dX^n_t=\int_Ub(t,X^n_t,\mu^{[n]}_t,u)\,\widetilde{Q}_t(du)dt+\int_U\sigma(t,X^n_t,\mu^{[n]}_t,u)\,\widetilde{M}(du,dt)+c(t)\,d\widetilde{Z}^{[n]}_t.
    \end{equation*}
For each $n$, set $\mathbb{P}^n=\widetilde{\mathbb{P}}\circ(X^n,\widetilde Q,\widetilde{Z})^{-1}$. It is easy to check that $\mathbb{P}^{n}\in\mathcal{R}^{[n]}(\mu^{[n]})$. Standard estimates yield,
\begin{equation}\label{second-order-moment-estimate-1}
\begin{split}
    E^{\widetilde{\mathbb{P}}}\int_0^T |X^n_t-\widetilde{X}_t|^2\,dt\leq&~ CE^{\widetilde{\mathbb{P}}}\int_0^T \left|\int_0^tc(s)\,d\widetilde{Z}^{[n]}_s-\int_0^tc(s)\,d\widetilde{Z}_s\right|^2\,dt\\
        &+~CE^{\widetilde{\mathbb{P}}}\int_0^T\left(1+L(W_p(\mu^{[n]}_t,\delta_0),W_p(\mu^*_t,\delta_0))\right)^2\mathcal{W}_p(\mu^{[n]}_t,\mu^*_t)^2\,dt.
\end{split}
\end{equation}
$\widetilde{Z}^{[n]}\rightarrow \widetilde{Z}$ in $M_1$ a.s. implies
\[
    E^{\widetilde{\mathbb{P}}}\int_0^T \left|\int_0^tc(s)\,d\widetilde{Z}^{[n]}_s-\int_0^tc(s)\,d\widetilde{Z}_s\right|^2\,dt\rightarrow 0.
\]
By the same arguments leading to (\ref{convergence-from-skorokhod}) in the proof of Lemma \ref{continuity-J-mu-P},
\[
    E^{\widetilde{\mathbb{P}}}\int_0^T\left(1+L(W_p(\mu^{[n]}_t,\delta_0),W_p(\mu^*_t,\delta_0))\right)^2\mathcal{W}_p(\mu^{[n]}_t,\mu^*_t)^2\,dt\rightarrow 0.
\]
This yields,
\begin{equation}\label{second-order-moment-estimate-2}
    \lim_{n\rightarrow\infty}E^{\widetilde{\mathbb{P}}}\int_0^T |X^n_t-\widetilde{X}_t|^2\,dt=0.
\end{equation}
Hence, up to a subsequence, dominated convergence implies
    \begin{equation*}
      \begin{split}
         \lim_{n\rightarrow\infty}J^{[n]}(\mu^{[n]},\mathbb{P}^n) =&~\lim_{n\rightarrow\infty}E^{\widetilde{\mathbb{P}}}\left[\int_0^T \int_U f(t,X^n_t,\mu^{[n]}_t,u)\,\widetilde{Q}_t(du)dt\right] \\
           = & E^{\widetilde{\mathbb{P}}}\left[\int_0^T\int_U f(t,X_t,\mu^*_t,u)\,\widetilde{Q}_t(du)dt\right] \\
           =  & J(\mu^*,\mathbb{P}).
       \end{split}
    \end{equation*}
Moreover, by Lemma \ref{continuity-J-mu-P}, $$\lim_{n\rightarrow\infty}J^{[n]}(\mu^{[n]},\mathbb{P}^{[n]})= J(\mu^*,\mathbb{P}^*).$$

Altogether, this yields,
\[J(\mu^*,\mathbb{P})=\lim_{n\rightarrow\infty}J^{[n]}(\mu^{[n]},\mathbb{P}^n)\geq\lim_{n\rightarrow\infty}J^{[n]}(\mu^{[n]},\mathbb{P}^{[n]})= J(\mu^*,\mathbb{P}^*).\]
\end{proof}


{\subsection{Approximating a given solutions to MFGs with singular controls}\label{Sec4.2}

In this subsection, we show how to approximate a  {\sl given} solution to a MFG with singular controls of the form \eqref{MFG1} introduced in the previous subsection by a sequence of admissible control rules of MFGs with only regular controls.

Let $\mathbb{P}^*$ be any solution to the MFG \eqref{MFG1}. Since $(\Omega,\{\mathcal{F}_t,t\in\mathbb R\},\mathbb{P}^*,X,Q,Z)$ satisfies the associated martingale problem, there exists a tuple $(\widehat{X},\widehat{Q},\widehat{Z},M)$ defined on some extension $(\widehat{\Omega},\{\widehat{\mathcal{F}}_t,t\in\mathbb R\},{\mathbb{Q}})$ of the canonical path space, such that
\[
    \mathbb{P}^*\circ(X,Q,Z)^{-1}=\mathbb{Q}\circ (\widehat{X},\widehat{Q},\widehat{Z})^{-1}
\]
and
\begin{equation}\label{X-hat-Z-hat}
    \mathbb{Q}\left(\widehat{X}_{\cdot}=\int_0^{\cdot}\int_U b(s,\widehat{X}_s,\mu^*_s,u)\,\widehat{Q}_s(du)ds+\int_0^{\cdot}\int_U\sigma(s,\widehat{X},\mu^*_s,u)\,M(du,ds)+\int_0^{\cdot}c(s)\,d\widehat{Z}_s\right)=1.
\end{equation}
Let $X^{[n]}$ be the unique strong solution of the SDE
\begin{equation}\label{X-n-Z-hat-n}
    dX^{[n]}_t=\int_U b(t,X^{[n]}_t,\mu^{[n]}_t,u)\,\widehat{Q}_t(du)dt+\int_U\sigma(t,X^{[n]}_t,\mu^{[n]}_t,u)\,M(du,dt)+c(t)\,d\widehat{Z}^{[n]}_t,
\end{equation}
where $\widehat{Z}^{[n]}$ is defined by (\ref{approximation-of-singular-control}) and $\mu^{[n]}$ is any sequence satisfying $\mu^{[n]}\rightarrow\mu^*$ in $\mathcal W_{p,(\widetilde{\mathcal D}(\mathbb R),d_{M_1})}$. One checks immediately that
\[
	\mathbb{P}^{[n]}:=\mathbb{Q}\circ (X^{[n]},\widehat{Q}, \widehat{Z})^{-1}\in\mathcal{R}^{[n]}(\mu^{[n]}).
\]
Our goal is to show that the sequence $\{\mathbb{P}^{[n]}\}_{n \geq 1}$ converges to $\mathbb{P}^*$ in $\mathcal{W}_p$ along some subsequence, which relies on the following lemma. 
 Its proof uses the notion of a parameter representation of the thin graph of a function $x \in {\cal D}(0,T)$ introduced in Appendix \ref{Appendix-C}.

\begin{proposition}\label{stability-SDE-M1}
On some probability space $(\Omega,\mathcal{F},\{\mathcal{F}_t,t\geq 0\},\mathbb{P})$, let $X^n$ and $X$ be the unique strong solution to SDE,
\begin{equation}
    dX^n_t=\int_U b(t,X^n_t,\mu^{n}_t,u)\,Q_t(du)dt+\int_U\sigma(t,X^n_t,\mu^n_t,u)\,M(du,dt)+\,dZ^n_t, ~t\in[0,\widetilde{T}]
\end{equation}
respectively,
\begin{equation}
    dX_t=\int_Ub(t,X_t,\mu_t,u)\,Q_t(du)dt+\int_U\sigma(t,X_t,\mu_t,u)\,M(du,dt)+\,dZ_t,  ~t\in[0,\widetilde{T}]
\end{equation}
where $\widetilde{T}$ is a fixed positive constant, $b$ and $\sigma$ satisfy $\mathcal{A}_1$ and $\mathcal{A}_5$.
If $Z^n\rightarrow Z$ in $(\mathcal{A}^m(0,\widetilde{T}),d_{M_1})$ a.s. and $\mu^n\rightarrow\mu$ in $\mathcal{W}_{p,(\mathcal{D}(0,\widetilde{T}),d_{M_1})}$, then
\[
   \lim_{n\rightarrow\infty} E^{\mathbb{P}}d_{M_1}(X^n,X)=0.
\]
\end{proposition}
\begin{proof}
By the a.s.~convergence of $Z^{n}$ to $Z$ in $M_1$, there exists $\underline{\Omega}\subseteq\Omega$ with full measure such that $d_{M_1}(Z^n(\omega),Z(\omega))\rightarrow 0$ for each $\omega\in\underline{\Omega}$. Furthermore, by Proposition \ref{characterization-convergence-M1}(2), for each $\omega\in\underline{\Omega}$, there exist parameter representations $(u(\omega),r(\omega))\in\Pi_{Z(\omega)}$ and $(u_n(\omega),r_n(\omega))\in\Pi_{Z^{n}(\omega)}$ of $Z(\omega)$ and $Z^n(\omega)$ $(n \in \mathbb{N})$, respectively, such that
    \begin{equation}\label{convergence-parameter-representation}
       \|u_n(\omega)-u(\omega)\|\rightarrow 0\textrm{ and }\|r_n(\omega)-r(\omega)\|\rightarrow 0.
    \end{equation}
Parameter representations with the desired convergence properties are
constructed in, e.g., \cite[ Section 4]{Pang-Whitt-2010}; see also \cite[ Theorem 1.2]{Pang-Whitt-2010}. 
A careful inspection of \cite[Section 4]{Pang-Whitt-2010} shows that the constructions of $(u(\omega),r(\omega))$ and $(u_n(\omega),r_n(\omega))$ only use measurable operations. As a result the mappings $(u(\cdot),r(\cdot))$ and $(u_n(\cdot),r_n(\cdot))$ are measurable. 

We now construct parameter representations $(u_{X^n}(\omega),r_{X^n}(\omega))$ and $(u_{X}(\omega),r_{X}(\omega))$ of $X^n(\omega)$ and $X(\omega)$, respectively. Since $X(\omega)$ (resp. $X^n(\omega)$) jumps at the same time as $Z(\omega)$ (resp. $Z^{n}(\omega)$), we can choose
\[
    r_{X}(\omega)=r(\omega),~~r_{X^n}(\omega)=r_n(\omega).
\]
In the following, we will drop the dependence on $\omega\in\underline{\Omega}$, if there is no confusion. By \cite[equation (3.1)]{Pang-Whitt-2010}, parameter representations of $X^n$ and $X$ in terms of the parameter representations of $Z^n$ and $Z$ are given by, respectively,
  \begin{equation*}
      u_{X^n}(t)=\int_0^{r_n(t)}\int_U b(s,X^n_s,\mu^{n}_{s},u)\,Q_s(du)ds+\int_0^{r_n(t)}\int_U \sigma(s,X^n_s,\mu^n_s,u)\,M(du,ds)+u_n(t),
    \end{equation*}
and
      \begin{equation*}
      u_{X}(t)=\int_0^{r(t)} \int_Ub(s,X_s,\mu_s,u)\,Q_s(du)ds+\int_0^{r(t)}\int_U \sigma(s,X_s,\mu_s,u)\,M(du,ds)+u(t).
    \end{equation*}
Hence, by the Lipschitz property of $b$ and $\sigma$ and BDG's inequality, we get,
\begin{equation}\label{u-X-n-minus-u-X}
\begin{split}
&~E\sup_{0\leq t\leq \widetilde T}|u_{X^n}(t)-u_{X}(t)|\\
        \leq&~CE\left(\int_0^{\widetilde{T}}|{X^n}(s)-{X}(s)|^2\,ds\right)^{\frac{1}{2}}\\
        &~+C\left(\int_0^{\widetilde T}\left(1+L(\mathcal{W}_p(\mu^n_{s},\delta_0),\mathcal{W}_p(\mu_{s},\delta_0))\right)^2\mathcal{W}^2_p(\mu^n_{s},\mu_{s})\,ds\right)^{\frac{1}{2}}
        +CE\sup_{0\leq t\leq \widetilde T}|r_n(t)-r(t)|\\
        &~+E\sup_{0\leq t\leq \widetilde T}\left|\int_0^{r_n(t)}\int_U \sigma(s,X_s,\mu_s,u)\,M(du,ds)-\int_0^{r(t)}\int_U \sigma(s,X_s,\mu_s,u)\,M(du,d{s})\right|\\
        &~+E\sup_{0\leq t\leq\widetilde T}\left|u_n(t)-u(t)\right|.
\end{split}
\end{equation}
The same argument as in the proof of Theorem \ref{limit-is-solution-MFG} yields that the first two terms on the right hand side of \eqref{u-X-n-minus-u-X} converge to $0$ while the last three terms converge to $0$ due to \eqref{convergence-parameter-representation}. Thus, 
\[
       \lim_{n\rightarrow\infty} E\sup_{0\leq t\leq \widetilde T}|u_{X^n}(t)-u_{X}(t)|=0.
\]

\end{proof}
\begin{corollary}
Under the assumptions of Proposition \ref{stability-SDE-M1}, along a subsequence $\mathbb{P}^{[n]}\rightarrow\mathbb{P}^*$ in $\mathcal{W}_p.$
\end{corollary}
\begin{proof}
{For each $\widetilde\epsilon>0$,} we extend the equations \eqref{X-hat-Z-hat} and \eqref{X-n-Z-hat-n} by
\[
    \widehat{X}_s=\int_{-\widetilde{\epsilon}}^s\int_U \widetilde{b}(t,\widehat{X}_t,\mu^*_t,u)\,\widehat{Q}_t(du)dt+\int_{-\widetilde{\epsilon}}^s\int_U \widetilde{\sigma}(t,\widehat{X}_t,\mu^*_t,u)\,M(du,dt)+\int_{-\widetilde{\epsilon}}^s\widetilde{c}(t)\,d\widehat{Z}_t,
\]
respectively,
\[
    X^{[n]}_s=\int_{-\widetilde{\epsilon}}^s\int_U \widetilde{b}(t,X^{[n]}_t,\mu^{[n]}_t,u)\,\widehat{Q}_t(du)dt+\int_{-\widetilde{\epsilon}}^s\int_U\widetilde{\sigma}(t,X^{[n]}_t,\mu^{[n]}_t,u)\,M(du,dt)+\int_{-\widetilde{\epsilon}}^s\widetilde{c}(t)\,d\widehat{Z}^{[n]}_t,
\]
where
\[
    \widetilde{b}(s,\cdot)=b(s,\cdot), ~\widetilde{\sigma}(s,\cdot)=\sigma(s,\cdot),~\widetilde{c}(s)=c(s)\textrm{ when }s\geq 0;~ \widetilde{b}(s,\cdot)=0, ~\widetilde{\sigma}(s,\cdot)=0,~\widetilde{c}(s)=c(0)\textrm{ when }s<0.
\]

Moreover, we have that $$\int_{-\widetilde{\epsilon}}^{\cdot}\widetilde{c}(t)\,d\widehat{Z}^{[n]}_t=\int_{-\widetilde{\epsilon}}^{\cdot}\widetilde{c}^+(t)\,d\widehat{Z}^{[n]}_t-\int_{-\widetilde{\epsilon}}^{\cdot}\widetilde{c}^-(t)\,d\widehat{Z}^{[n]}_t,$$ where a.s. in $({\mathcal{A}}^m(-\widetilde{\epsilon},T+\epsilon),d_{M_1})$,
\[
	\int_{-\widetilde{\epsilon}}^{\cdot}\widetilde{c}^+(t)\,d\widehat{Z}^{[n]}_t\rightarrow\int_{-\widetilde{\epsilon}}^{\cdot}\widetilde{c}^+(t)\,d\widehat{Z}^{}_t
	\quad \mbox{and} \quad
	\int_{-\widetilde{\epsilon}}^{\cdot}\widetilde{c}^-(t)\,d\widehat{Z}^{[n]}_t\rightarrow\int_{-\widetilde{\epsilon}}^{\cdot}\widetilde{c}^-(t)\,d\widehat{Z}^{}_t.
\]	
	Since $\int_{-\widetilde{\epsilon}}^{\cdot}\widetilde{c}^+(t)\,d\widehat{Z}^{}_t$ and $\int_{-\widetilde{\epsilon}}^{\cdot}\widetilde{c}^-(t)\,d\widehat{Z}^{}_t$ never jump at the same time, Proposition \ref{M1-continuity-addition} implies that
\[
    \int_{-\widetilde{\epsilon}}^{\cdot}\widetilde{c}(t)\,d\widehat{Z}^{[n]}_t\rightarrow \int_{-\widetilde{\epsilon}}^{\cdot}\widetilde{c}(t)\,d\widehat{Z}^{}_t
\]
a.s.~in $({\mathcal{A}}^m(-\widetilde{\epsilon},T+\epsilon),d_{M_1})$.
Hence, by Proposition \ref{stability-SDE-M1}, $$E^{\mathbb{Q}}d_{M_1}(X^{[n]},\widehat{X})\rightarrow 0.$$
 Hence, up to a subsequence,
 \[
    d_{M_1}(X^{[n]},\widehat{X})\rightarrow 0 \mbox{ in } \mathcal D(-\widetilde{\epsilon},T+\epsilon); \quad \mathbb{Q}\mbox{-a.s.},
 \]
 which implies the same convergence holds in $\widetilde{\mathcal D}_{0,T+\epsilon}(\mathbb R)$. For any nonnegative continuous function $\phi$ satisfying
\[
	\phi(x,q,z)\leq C(1+d_{M_1}(x,0)^p+\mathcal W_p^p(q/T,\delta_0)+d_{M_1}(z,0)^p),
\]	
 the uniform integrability of $d_{M_1}(X^{[n]},0)^p$, $\mathcal W_p^p(\widehat{Q}/T,\delta_0)$ and $d_{M_1}(\widehat{Z},0)^p$ yields $$E^{\mathbb{Q}}\phi(X^{[n]},\widehat{Q},\widehat{Z})\rightarrow E^{\mathbb{Q}}\phi(\widehat{X},\widehat{Q},\widehat{Z}).$$ This implies $\mathbb{Q}\circ (X^{[n]},\widehat{Q},\widehat{Z})^{-1}\rightarrow\mathbb{Q}\circ(\widehat{X},\widehat{Q},\widehat{Z})^{-1}$ in $\mathcal{W}_{p,\Omega}$ by \cite[Definition 6.8]{Villani-2009}, that is, $\mathbb{P}^{[n]}\rightarrow\mathbb{P}^*$ in $\mathcal{W}_{p,\Omega}$.
\end{proof}
}
\begin{appendix}
\section{Wasserstein distance and representation of martingales}\label{AA}

\begin{definition}\label{Wasserstein-space-metric}
Let $(E,\varrho)$ be a metric space. Denote by $\mathcal{P}_p(E)$ the class of all probability measures on $E$ with finite moment of $p$-th order.
The {\sl $p$-th Wasserstein metric} on $\mathcal{P}_p(E)$ is defined by:
\begin{equation}
\mathcal{W}_{p,(E,\varrho)}(\mathbb{P}_1,\mathbb{P}_2)=\inf\left\{\left(\int_{E\times E}\varrho(x,y)^p\,\gamma(dx,dy)\right)^{\frac{1}{p}}:\gamma(dx,E)=\mathbb{P}_1(dx),\gamma(E,dy)=\mathbb{P}_2(dy)\right\}.
\end{equation}
The set  $\mathcal{P}_p(E)$ endowed with the Wasserstein distance is denoted by $\mathcal{W}_{p,(E,\varrho)}$ or $\mathcal{W}_{p,E}$ or $\mathcal{W}_{p}$ if there is no risk of confusion about the underlying state space or distance.
\end{definition}


It is well known \cite[Theorem III-10]{Karoui-Meleard-1990} that for every continuous square integrable martingale $m$ with quadratic variation process $\int_0^{\cdot}\int_U a(s,u)\,v_s(du)ds$, where $a=\sigma\sigma^{\top}$ and $\sigma$ is a bounded measurable function and $v$ is $\mathcal{P}(U)$ valued stochastic process, on some extension of the original probability space, there exists a martingale measure $M$ with intensity $v_s(du)ds$ such that $m_{\cdot}=\int_0^{\cdot}\int_U \sigma(t,u)\,M(du,dt)$. This directly leads to the following proposition, which is frequently used in the main text.

\begin{proposition}\label{result-from-Karoui}
The existence of solution $\mathbb{P}$ to the martingale problem (\ref{martingale-problem-MFGs}) is equivalent to the existence of the weak solution to the following SDE
    \begin{equation}\label{martingale-measure-SDE}
      d\bar{X}_t=\int_U b(t,\bar{X}_t,\mu_t,u)\,\bar{Q}_s(du)ds+\int_U \sigma(t,\bar{X}_t,\mu_t,u)\,\bar{M}(du,dt)+c(t)\,d\bar{Z}_t,
    \end{equation}
where $\bar{X}$, $\bar{M}$ and $\bar{Z}$ are defined on some extension $(\bar{\Omega},\bar{\mathcal{F}},\bar{\mathbb{P}})$ and $\bar{M}$ is a martingale measure with intensity $\bar{Q}$. Moreover, the two solutions are related by $\mathbb{P}=\bar{\mathbb{P}}\circ (\bar{X},\bar{Q},\bar{Z})^{-1}$.
\end{proposition}
\section{Strong $M_1$ Topology in Skorokhod Space} \label{Appendix-C}
In this section, we summarise some definitions and properties about strong Skorokhod $M_1$ topology. For more details, please refer to Chapter 3, 11 and 12 in \cite{Whitt-2002}. Note that in \cite{Whitt-2002} two $M_1$ topologies are introduced, the strong one and the weak one. In this paper, we only apply the strong one. So without abuse of terminologies, we just take $M_1$ topology for short.

For $x\in \mathcal{D}(0,T)$, denote by $Disc(x)$ the set of discontinuous points of $x$. Note that on $[0,T]$, $Disc(x)$ is at most countable. Define the thin graph of $x$ as
\begin{align}\label{thin-graph}
G_x=\{(z,t)\in\mathbb{R}^d\times[0,T]:z\in[x_{t-},x_t]\},
\end{align}
where $x_{t-}$ is the left limit of $x$ at $t$ and $[a,b]$ means the line segment between $a$ and $b$, i.e., $[a,b]=\{\alpha a+(1-\alpha)b:0\leq\alpha\leq 1\}$. On the thin graph, we define an order relation. For each pair $(z_i,t_i)\in G_x$, $i=1,2$, $(z_1,t_1)\leq (z_2,t_2)$ if either of the following holds: (1) $t_1<t_2$; (2) $t_1=t_2$ and $|z_1-x_{t_1-}|<|z_2-x_{t_2-}|$.

Now we define the parameter representation, on which the $M_1$ topology depends. The mapping pair $(u,r)$ is called a parameter representation if $(u,r):[0,1]\rightarrow G_x$, which is continuous and nondecreasing w.r.t. the order relation defined above. Denote by $\Pi_x$ all the parameter representations of $x$. Let
\begin{align}\label{strong-M1-metric}
d_{M_1}(x_1,x_2)=\inf_{(u_i,r_i)\in \Pi_{x_i},i=1,2} ||u_1-u_2||\vee||r_1-r_2||.
\end{align}
It can be shown that $d_{M_1}$ is a metric on $\mathcal{D}(0,T)$ such that $\mathcal{D}(0,T)$ is a Polish space. The topology induced by $d_{M_1}$ is called $M_1$ topology.

For each $t\in[0,T]$ and $\delta>0$, the oscillation function around $t$ is defined as
\begin{align}\label{oscillation-fun}
\bar{v}(x,t,\delta)=\sup_{0\vee(t-\delta)\leq t_1\leq t_2 \leq(t+\delta)\wedge T}|x_{t_1}-x_{t_2}|,
\end{align}
and the so called strong $M_1$ oscillation function is defined as
\begin{align}\label{strong-M1-oscillation-fun}
w_s(x,t,\delta)=\sup_{0\vee(t-\delta)\leq t_1< t_2<t_3\leq (t+\delta)\wedge T}|x_{t_2}-[x_{t_1},x_{t_3}]|,
\end{align}
where $|x_{t_2}-[x_{t_1},x_{t_3}]|$ is the distance from $x_{t_2}$ to the line segment $[x_{t_1},x_{t_3}]$.
Moreover,
\begin{align}
w_s(x,\delta):=\sup_{0\leq t\leq T}w_s(x,t,\delta).
\end{align}


\begin{proposition}\label{characterization-convergence-M1}
The following statements about the characterization of $M_1$ convergence are equivalent,
    \begin{itemize}
    \item[1. ]  $x^n\rightarrow x$ in $M_1$ topology;
    \item[2. ] there exist $(u,r)\in\Pi_x$ and $(u^n,r^n)\in\Pi_{x^n}$ for each $n$ such that
        \[
        \lim\limits_{n\rightarrow \infty}\|u^n-u\|\vee \|r^n-r\|=0;
        \]
    \item[3. ] $x_n(t)\rightarrow x(t)$ for each $t\in [0,T]\setminus Disc(x)$ including $0$ and $T$, and
        \[
         \lim_{\delta\rightarrow 0}\overline{\lim}_{n\rightarrow\infty}w_s(x^n,\delta)=0.
         \]
    \end{itemize}
Moreover, each one of the above three items implies the local uniform convergence of $x^n$ to $x$ at each continuous point of $x$, that is, for each $t\not\in Disc(x)$, there holds
\begin{equation}\label{local-uniform-convergence-at-continuity}
\lim_{\delta\rightarrow 0}\limsup_{n\rightarrow\infty}\sup_{t-\delta\leq s\leq t+\delta}|x_n(s)-x(s)|=0.
\end{equation}
\end{proposition}
\begin{remark}\label{M1-stronger-L1}
\begin{itemize}
\item[(1)] When restricted to $\mathcal C(0,T)$, the uniform topology is equivalent to the $M_1$ topology. Indeed, when $x^n\in\mathcal C(0,T)$ and $\|x^n-x\|\rightarrow 0$, then $x\in\mathcal C(0,T)$. For any $r:[0,1]\rightarrow[0,T]$, define
    \[
        u^n(t):=x^n_{r(t)},\qquad u(t):=x_{r(t)}.
    \]
    Thus, $(u^n,r)$ and $(u,r)$ can serve as a parameter representation of $x^n$ and $x$, respectively. Moreover, it holds that $\|u^n-u\|\rightarrow 0$. So we have $x^n\rightarrow x$ in $M_1$ by Proposition \ref{characterization-convergence-M1}(2). On the other hand, when $x^n,x\in\mathcal C(0,T)$ and $x^n\rightarrow x$ in $M_1$, by Proposition \ref{characterization-convergence-M1}(2), there exist parameter representations $(u^n,r^n)$ and $(u,r)$ of $x^n$ and $x$, respectively, such that
    \[
        u^n(t)=x^n_{r^n(t)}, ~u(t)=x_{r(t)},~\|u^n-u\|\rightarrow 0\textrm{ and }\|r^n-r\|\rightarrow 0.
    \]
    To show that $\|x^n-x\|\rightarrow 0$, it is sufficient to prove that $t^n\rightarrow t$ implies $|x^n_{t^n}-x_t|\rightarrow 0$. Let $r^n(s^n)=t^n$ and $r(s)=t$. Then we have
    \[
        |r(s^n)-r(s)|\leq |r^n(s^n)-r(s^n)|+|r^n(s^n)-r(s)|\rightarrow 0.
    \]
    So we get
    \[
        |x^n_{t^n}-x_t|=|u^n(s^n)-u(s)|\leq |u^n(s^n)-u(s^n)|+|u(s^n)-u(s)|=|u^n(s^n)-u(s^n)|+|x_{r(s^n)}-x_{r(s)}|\rightarrow 0.
    \]
\item[(2)]
Proposition \ref{characterization-convergence-M1}(3) implies that $(\mathcal{D}(0,T),d_{M_1})$ convergence is stronger than $L^{\alpha}[0,T]$ convergence, for any $\alpha>0$. In fact, if $x^n\rightarrow x$ in $M_1$, then $x^n_t\rightarrow x_t$ for a.e. $t\in[0,T]$, due to Proposition \ref{characterization-convergence-M1}(3). Moreover,
\[
    |x^n_t-x_t|^{\alpha}\,\leq 2^{\alpha}\left(d_{M_1}^{\alpha}(x^n,0)+d_{M_1}^{\alpha}(x,0)\right)\rightarrow 2^{\alpha+1}d_{M_1}(x,0)<\infty.
\]
Thus, the assertion follows from dominated convergence.
\end{itemize}
\end{remark}
\begin{proposition}\label{characterization-relative-compactness-strong-M1}
A subset $A$ of $(\mathcal{D}(0,T), d_{M_1})$ is relatively compact w.r.t. $M_1$ topology if and only if
     \begin{align}\label{uniform-boundedness}
        \sup_{x\in A}||x||<\infty
     \end{align}
    and
     \begin{align}\label{oscillation-condition-M1-compactness}
        \lim_{\delta \downarrow 0}\sup_{x\in A}w'_s(x,\delta)=0,
     \end{align}
    where
    \begin{align}
        w'_s(x,\delta)=w_s(x,\delta)\vee\bar{v}(x,0,\delta)\vee\bar{v}(x,T,\delta).
   \end{align}
\end{proposition}

In \cite{Whitt-2002}, it is assumed that $x_{0-}=x_0$, which implies there is no jump at the initial time. For singular control problems it is natural to admit jumps a the initial time. It is also implied by Proposition \ref{characterization-relative-compactness-strong-M1} that the terminal time $T$ is a continuous point of $x\in\mathcal{D}(0,T)$. This, too, is not appropriate for singular control problems. In order to adapt the relative compactness criteria stated in Proposition \ref{characterization-relative-compactness-strong-M1} to functions with jumps at $0$ and $T$, we work on the extended state spaces $\widetilde{\mathcal{D}}(\mathbb{R})$ and $\widetilde{\mathcal{A}}(\mathbb{R})$. Convergence in $\widetilde{\mathcal{D}}(\mathbb{R})$ can be defined as convergence in $\mathcal{D}(\mathbb{R})$, where a sequence $\{x^n, n \geq 1\}$ converges to $x$ in $\mathcal{D}(\mathbb{R})$ if and only if the sequences $\{x^n|_{[a,b]}, n \geq 1\}$ converge to $x|_{[a,b]}$ for all $a<b$ at which $x$ is continuous; see \cite[Chapter 3]{Whitt-2002}.

Relative compactness of a sequence $\{x^n,n\geq 1\}\subseteq \widetilde{\mathcal{D}}(\mathbb{R})$ is equivalent to that of the sequence $\{x^n|_{[a,b]},n\geq 1\}\subseteq \mathcal{D}[a,b]$ for any $a<0$ and $b>T$. Specifically, we have the following result.

\begin{proposition}\label{modified-M1}
The sequence $\{x^n,n\geq 1\}\subseteq \widetilde{\mathcal{D}}(\mathbb{R})$ is relatively compact if and only if
 \begin{equation}\label{modified-relative-compactness}
 \begin{split}
        \sup_{n}||x_n||<\infty \quad \mbox{and} \quad \lim_{\delta \downarrow 0}\sup_{x\in A}\widetilde{w}_s(x,\delta)=0,
 \end{split}
 \end{equation}
where the modified oscillation function $\widetilde{w}_s$ is defined as
\begin{equation}\label{modified-oscillation}
    \widetilde{w}_s(x,\delta)={w}_s(x,\delta)+\sup_{0\leq s<t\leq \delta}|x_s-[0,x_t]|.
\end{equation}
\end{proposition}

We notice that the modified oscillation function $\widetilde{w}_s$ is defined in terms of the original oscillation function ${w}_s$ and the line segment (if it exists) between $0-$ and $0$\footnote{Due to the right-continuity of the elements in $\widetilde{\mathcal{D}}(\mathbb{R})$ there is no line segment between $T$ and $T+$.}. 
\begin{corollary}\label{compactness-finite-fuel}
Let $A=\{z\in\widetilde{\mathcal{A}}(\mathbb{R}):z_T\leq K\}$ for some $K>0$. Then $A$ is $(\widetilde{\mathcal{D}}(\mathbb{R}),M_1)$ compact.
\end{corollary}
\begin{proof}
This follows from Proposition \ref{modified-M1} as $w_s(z,t,\delta)=0$ for each $z\in A$, $t \in \mathbb{R}$ and $\delta > 0$.
\end{proof}
%
\begin{proposition}\label{modified-tightness-criterion}
A sequence of probability measures $\{\mathbb{P}_n\}_{n\geq 1}$ on $\widetilde{\mathcal{D}}(\mathbb{R})$ is tight if and only if
\newline(1) for each $\epsilon>0$, there exists $c$ large enough such that
   \begin{align}\label{tightness-criterion-uniform-boundedness}
      \sup_n\mathbb{P}_n(||x||>c)<\epsilon;
   \end{align}
(2) for each $\epsilon>0$ and $\eta>0$, there exists $\delta>0$ small enough such that
    \begin{align}\label{tightness-criterion-modified-oscillation}
      \sup_n\mathbb{P}_n(\widetilde{w}_s(x,\delta)\geq \eta)<\epsilon.
    \end{align}
\end{proposition}
The following proposition shows that if two $M_1$ limits do not jump at the same time, then the $M_1$ convergence preserves by the addition operation.
\begin{proposition}\label{M1-continuity-addition}
If $x^n\rightarrow x$ and $y^n\rightarrow y$ in $(\mathcal{D}(0,T), d_{M_1})$, and $Disc(x)\cap Disc(y)={\O}$, then
   \begin{align}
      x^n+y^n\rightarrow x+y~in~M_1.
   \end{align}
\end{proposition}
\section{Sketch Proof of Proposition \ref{transfer-X-to-Y}}\label{sketch-proof}
 It is sufficient to establish the equivalence of martingale problems in Definition \ref{control-rule-mu} and Proposition \ref{transfer-X-to-Y}. Only the one-dimensional case is proved; the multi-dimensional case is similar.

 \vspace{1mm}
        Proposition \ref{transfer-X-to-Y} $\Rightarrow$ Definition \ref{control-rule-mu}: Without loss of generality (see \cite[Proposition 4.11 and Remark 4.12]{Karatzas-Shreve-1991}), we can take $\phi(y)=y,~y^2$ and following the proof of \cite[Proposition 4.6]{Karatzas-Shreve-1991}, we have that $M$ is a continuous martingale with the quadratic variation
        \[
            \langle M\rangle_t=\int_0^t\int_U a(s,X_s,\mu_s,u)\,Q_s(du)ds,
        \]
        where
        \[
            M_t=Y_t-\int_0^t\int_U b(s,X_s,\mu_s,u)\,Q_s(du)ds.
        \]
        By applying It\^o's formula to $\phi(X_t)$ and noting $X=Y+\int_0 c(s)\,dZ_s$, the desired result follows from
        \begin{equation*}
        \begin{split}
            \phi(X_t)=&\phi(X_{0-})+\int_0^t\int_U\phi'(X_s)b(s,X_s,\mu_s,u)\,Q_s(du)ds+\frac{1}{2}\int_0^t\int_U\phi''(X_s)a(s,X_s,\mu_s,u)\,Q_s(du)ds\\
            &+\int_0^t\phi'(X_{s-})c(s)\,dZ_s+\sum_{0\leq s\leq t}\left[\phi(X_s)-\phi(X_{s-})-\phi'(X_{s-})\triangle X_s\right]+\int_0^t\phi'(X_s)\,dM_s.
        \end{split}
        \end{equation*}
        \newline Definition \ref{control-rule-mu} $\Rightarrow$ Proposition \ref{transfer-X-to-Y}: By Proposition \ref{result-from-Karoui}, there exits $(\overline{X},\overline{Q},\overline{Z})$ and a martingale measure $\overline{M}$ with intensity $\overline{Q}$ on some extension $(\overline{\Omega},\overline{\mathcal{F}},\overline{\mathbb{P}})$, s.t.~\eqref{martingale-measure-SDE} holds and $\mathbb{P}\circ(X,Q,Z)^{-1}=\overline{\mathbb{P}}\circ (\overline{X},\overline{Q},\overline{Z})^{-1}$. Let
        \[
            \overline{Y}_{\cdot}=\overline{X}_{\cdot}-\int_0^{\cdot}c(s)\,d\overline{Z}_s.
        \]
        Then
        \[
            {Y}_{\cdot}:={X}_{\cdot}-\int_0^{\cdot}c(s)\,d{Z}_s\overset{d}=\overline{Y}.
        \]
        By applying It\^o's formula to $\phi(\overline{Y}_t)$,
        \[
            \phi(\overline{Y}_t)-\int_0^t\int_U\phi'(\overline{Y}_s)b(s,\overline{X}_s,\mu_s,u)\,\overline{Q}_s(du)ds-\frac{1}{2}\int_0^t\int_U\phi''(\overline{Y}_s)a(s,\overline{X}_s,\mu_s,u)\,\overline{Q}_s(du)ds
        \]
        is a martingale. Hence the following is also a martingale:
        \[
            \phi({Y}_t)-\int_0^t\int_U\phi'({Y}_s)b(s,{X}_s,\mu_s,u)\,{Q}_s(du)ds-\frac{1}{2}\int_0^t\int_U\phi''({Y}_s)a(s,{X}_s,\mu_s,u)\,{Q}_s(du)ds.
        \]
\end{appendix}

\end{document}